\documentclass[11pt]{amsart}
\usepackage{enumerate}
\usepackage{a4wide}

\usepackage{graphicx}
\usepackage{amsmath}
\usepackage{amsfonts}
\usepackage{amsthm}
\usepackage{url}
\usepackage{amsopn,amssymb,mathrsfs}
\usepackage{tikz}
\usepackage{booktabs}
\usepackage[margin=1in]{geometry}
\usepackage{parskip}
\usepackage{fancyvrb}
\usepackage{listings}

\newcounter{ctr}

\newcounter{ctr1}

\newcounter{ctr2}

\newcounter{ctr3}

\newtheorem{definition}{Definition}[section] 
\newtheorem{theorem}[definition]{Theorem}
\newenvironment{theorem*}[1]{{\bf Theorem #1} \begin{itshape}}{\end{itshape}}
\newtheorem{lemma}[definition]{Lemma}
\newtheorem{corollary}[definition]{Corollary}
\newenvironment{corollary*}[1]{{\bf Corollary #1} \begin{itshape}}{\end{itshape}}
\newtheorem{proposition}[definition]{Proposition}
\newenvironment{proposition*}[1]{{\bf Proposition #1} \begin{itshape}}{\end{itshape}}
\newtheorem{remark}[definition]{Remark}

\newcommand{\al}{\alpha}

\newcommand{\ud}{\, {\rm d} \kern-.015em }

\newcommand{\modulus}[1]{\left| \kern.05em #1 \kern.05em \right|}
\newcommand{\norm}[1]{\left\| \kern.05em #1 \kern.05em \right\|}
\newcommand{\inner}[1]{\left\langle \kern.05em #1 \kern.05em \right\rangle }

\newcommand{\bm}[1]{\mbox{\protect\boldmath $ #1 $}}

\newcommand{\pick}[2]{\renewcommand{\arraystretch}{0.6}
\left( \kern-.4em \begin{array}{c} #1 \\ #2 \end{array} \kern-.4em \right) }

\newcommand{\DQE}{\text{DQE}}

\newsavebox{\FVerbBox}

\newcommand{\ha}{\mathcal{H}}

\newcommand{\Q}{\mathbb{Q}}
\newcommand{\U}{\mathcal{U}}
\newcommand{\R}{\mathbb{R}}
\newcommand{\N}{\mathbb{N}}
\newcommand{\Z}{\mathbb{Z}}
\newcommand{\bx}{\boldsymbol{x}}
\newcommand{\btau}{\boldsymbol{\tau}}
\newcommand{\bsigma}{\boldsymbol{\sigma}}
\newcommand{\bt}{\boldsymbol{t}}

\newcommand{\m}{\mathcal{M}}

\newcommand{\bp}{\boldsymbol{p}}
\newcommand{\ba}{\boldsymbol{a}}
\newcommand{\Cf}{\mathcal{C}_{\mathbf{f}}}

\newcommand{\Zp}{\mathbb{Z}_{p}}
\newcommand{\Qp}{\mathbb{Q}_{p}}
\newcommand{\W}{\mathcal{W}}
\newcommand{\A}{\mathcal{A}}

\newcommand{\Ap}{\mathfrak{A}}

\newcommand{\Wp}{\mathfrak{W}}
\newcommand{\Sp}{\Wp}
\newcommand{\h}{\mathcal{H}}

\newcommand{\by}{\boldsymbol{y}}

\newcommand{\bal}{\boldsymbol{\alpha}}
\newcommand{\bv}{\mathbf{v}}
\newcommand{\bff}{\mathbf{f}}
\newcommand{\ff}{\mathbf{F}}
\begin{document}

\title{Simultaneous $p$-adic Diophantine approximation}

\author{V. Beresnevich}
\address[V. Beresnevich]{Department of Mathematics, University of York, Heslington, York, YO10
5DD, United Kingdom}
\email{victor.beresnevich@york.ac.uk}

\author{J. Levesley}
\address[J. Levesley]{Department of Mathematics, University of York, Heslington, York, YO10
5DD, United Kingdom}
\email{jason.levesley@york.ac.uk}

\author{B. Ward}
\address[B. Ward]{Department of Mathematics, University of York, Heslington, York, YO10
5DD, United Kingdom}
\email{bw744@york.ac.uk}
\date{\today}


\begin{abstract}
The goal of this paper is to develop the theory of weighted Diophantine approximation of rational numbers to $p$-adic numbers. Firstly, we establish complete analogues of Khintchine's theorem, the Duffin-Schaeffer theorem and the Jarn\'ik-Besicovitch theorem for `weighted' simultaneous Diophantine approximation in the $p$-adic case. Secondly, we obtain a lower bound for the Hausdorff dimension of weighted simultaneously approximable points lying on $p$-adic manifolds. This is valid for very general classes of curves and manifolds and have natural constraints on the exponents of approximation. The key tools we use in our proofs are the Mass Transference Principle, including its recent extension due to Wang and Wu \cite{WW19}, and a Zero-One law for weighted $p$-adic approximations established in this paper.
\end{abstract}

\maketitle

\section{Introduction}
\label{Intro}

One of the central themes in the theory of Diophantine approximation is to understand how rational points simultaneously approximate several given numbers. In this paper we will investigate simultaneous rational approximations to $p$-adic numbers. To begin with, we give a brief overview of relevant results in the real case, which is far better understood. Throughout  $\Psi=(\psi_{1}, \dots , \psi_{n})$ will be an $n$-tuple of approximation functions $\psi_{i}: \R_+ \to \R_{+}$ $(1 \leq i \leq n)$ such that $\psi_i(q) \to 0$ as $ q\to\infty$. Here and elsewhere $\R_+$ is the set of positive real numbers. Given any $\Psi$ as above and $q \in \N$, let
\begin{equation*}
\A_{q}(\Psi)= \bigcup_{\bp=(p_1,\dots,p_n) \in\Z^n} \left\{ \bx=(x_1,\dots,x_n) \in \R^{n} : \left| x_{i}- \frac{p_{i}}{q} \right| < \frac{\psi_{i}(q)}{q}\quad \text{for all } 1 \leq i \leq n \right\}
\end{equation*}
and
\begin{equation*}
\W_{n}(\Psi)= \limsup_{q \to \infty} \A_{q}(\Psi).
\end{equation*}
Thus $\W_{n}(\Psi)$ is the set of points $\bx$ such that there are infinitely many rational points $\bp/q$ that approximate $\bx$ with the error $\psi_i(q)/q$ in the $i$th coordinate.
In the case all approximation functions in $\Psi$ are the same, that is $\psi_{1}= \dots= \psi_{n}$, $\W_{n}(\Psi)$ is the standard set of simultaneously $\psi$-approximable points in $\R^n$, in which case we will write $\W_{n}(\psi)$ for $\W_n(\Psi)$. If the approximation functions have the form $\psi_{i}(q)=q^{-\tau_{i}}$ for some {\em exponents of approximation} $\btau=(\tau_{1}, \dots \tau_{n}) \in \R^{n}_{+}$ we will use the notation $\W_{n}(\btau)$ for $\W_{n}(\Psi)$. Furthermore, if $\btau=(\tau,\dots,\tau)$ for some $\tau>0$ we will write $\W_{n}(\tau)$ for $\W_{n}(\btau)$. Note that $\W_{n}(\Psi)+\ba=\W_{n}(\Psi)$ for any $\ba\in\Z^n$ and therefore $\W_{n}(\Psi)$ is often restricted to $[0,1]^n$ for convenience.

The following is a well know result that was originally proved by Khintchine \cite{K26} in 1926 when $\psi_{1}= \dots= \psi_{n}$ with a slightly more restrictive condition on the approximation function and which can be found in \cite{Harman} in full generality.

\begin{theorem} \label{khintch}
Let $\psi_{i}:\N \to \R_{+}$ be monotonically decreasing functions for each $1 \leq i \leq n$. Then
\begin{equation*}
\lambda_{n}\left(\W_{n}(\Psi)\cap[0,1]^n\right)= \begin{cases}
0  &\text{ if } \, \, \sum_{q=1}^{\infty} \prod_{i=1}^{n}\psi_{i}(q) < \infty, \\
1 & \text{ if } \, \, \sum_{q=1}^{\infty} \prod_{i=1}^{n}\psi_{i}(q) = \infty,
\end{cases}
\end{equation*}
where  $\lambda_{n}$ is $n$-dimensional Lebesgue measure.
\end{theorem}

To gather more precise information about the sets of measure zero arising from Theorem~\ref{khintch} one often appeals to Hausdorff measures and Hausdorff dimension. We now briefly recall these notions.
Let  $(X,d)$ be a separable metric space and suppose that $U \subseteq X$. For any $\rho>0$, a $\rho$-cover of $U$ is a countable collection of balls $\{B_{i}\}$  of radii $ r_i > 0$ such that $U \subset \cup_{i}B_{i}$ and $r_i \leq \rho$ for all $i$. A dimension function $f : \R_{+} \to \R_{+}$ is an increasing continuous function for which $f(r) \to 0$ as $r \to 0$. Define
\begin{equation*}
\ha^{f}_{\rho}(U)= \inf \left\{ \sum_{i}f(r_i): \{ B_{i} \} \, \, \text{ is a $\rho$-cover of } U \right\},
\end{equation*}
where the infimum is taken over all $\rho$-covers of $U$.  The Hausdorff $f$-measure, $\ha^{f}(U)$, of $U$ is defined as
\begin{equation*}
\ha^{f}(U)=\lim_{\rho \to 0^{+}} \ha^{f}_{\rho}(U).
\end{equation*}
When the dimension function $f(x)=x^{s}$ we will write $\ha^{s}$ for $\ha^{f}$. The Hausdorff dimension of $U$ is defined as
\begin{equation*}
\dim U = \inf \{ s\geq 0: \ha^{s}(U)= 0 \}.
\end{equation*}
Regarding the set $\W_{n}(\btau)$ Rynne proved the following general statement \cite{R98}.

\begin{theorem} \label{rynne}
Let $\btau=(\tau_{1}, \dots, \tau_{n}) \in \R^{n}_{+}$ be such that $\tau_1\ge \dots \ge \tau_n$ and $\sum_{i=1}^{n} \tau_{i} \geq 1$. Then
\begin{equation*}
\dim \W_{n}(\btau)= \min_{1 \leq k \leq n} \left\{ \frac{n+1+\sum_{i=k}^{n}(\tau_{k}-\tau_{i})}{\tau_{k}+1} \right\}.
\end{equation*}
\end{theorem}

Note that the condition $\sum_{i=1}^{n} \tau_{i} \geq 1$ on the exponents is optimal since, by Dirichlet's theorem, $\W_{n}(\btau)=\R^n$ whenever $\sum_{i=1}^{n} \tau_{i} \leq 1$.
The version of Theorem~\ref{rynne} with $ \tau_1 = \dots = \tau_n$ is a classical result independently established by Jarn\'ik \cite{J29} and Besicovitch \cite{B34}. Furthermore, Jarn\'ik proved a stronger Hausdorff measure result \cite{J29}.

When the coordinates of $\bx$ are confined by some functional relations, that is they are dependent, we fall into the area of Diophantine approximation on manifolds, see \cite{BDV07, BRV16, BD99, KM98, S80} for a general introduction. In this context, generalisations of Khintchine's theorem and the Jarn\'ik-Besicovitch theorem to submanifolds of $\R^n$ have been studied in great depth. Indeed, the theory for non-degenerate planar curves was essentially completed in \cite{BDV07,MR2285737,VV06}, also see \cite{MR2347267,BZ10,MR2777039,MR3318157,MR4092233} for subsequent developments. Note that if a planar curve is non-degenerate everywhere it means that the curvature of the curve is defined and non-zero everywhere except possibly at a countable number of points. In relation to manifolds in higher dimensions non-degeneracy implies that the manifolds are sufficiently curved so as to deviate from any hyperplane, see \cite{KM98} for a formal definition. The divergence case of Khintchine's theorem was obtained for arbitrary analytic non-degenerate manifolds in \cite{B12} and for arbitrary (not necessarily analytic) non-degenerate curves \cite{BVVZ}. The convergence case was obtained for various classes of manifolds, see \cite{BVVZ17, HL18, H20, R15, S18}
and references within, before being fully proven in \cite{BY21}. Note that the fact that the points $\bx$ lie on a submanifold of $\R^n$ implies that approximating rational points have to lie close to the manifold. Note that in \cite{BY21} a count on the number of rational points lying close to non-degenerate manifolds was established with the error term estimating the measure of $\psi$-approximable points as opposed to counting rational points that contribute to the error term.


In relation to Hausdorff dimension, the Jarn\'ik-Besicovitch theorem with weights was obtained in \cite{MR2285737} for non-degenerate planar curves. This reads as follows.

\begin{theorem} \label{euclid:curves}
Let $f \in C^{(3)}(I_{0})$, where $I_{0} \in \R$ is an interval, and $\Cf:=\{ (x, f(x)): x \in I_{0} \}$. Let $\btau=(\tau_{1},\tau_{2})$ be an exponent vector with
\begin{equation*}
0 < \min \{ \tau_{1},\tau_{2}\}<1 \, \text{ and } \, \tau_{1}+\tau_{2} \geq 1,
\end{equation*}
and assume that
\begin{equation} \label{curve_condition}
\dim \left\{ x \in I_{0}: f''(x)=0 \right\} \leq \frac{2- \min \{ \tau_{1}, \tau_{2} \}}{1+\max \{ \tau_{1},\tau_{2} \}}.
\end{equation}
Then
\begin{equation*}
\dim \W_{2}(\btau) \cap \Cf = \frac{2- \min \{ \tau_{1}, \tau_{2} \}}{1+\max \{ \tau_{1},\tau_{2} \}}.
\end{equation*}
\end{theorem}

In establishing the upper bound of Theorem~\ref{euclid:curves}, an estimate due to Huxley \cite{H96} is used as the key ingredient of the proof. To be more specific, Huxley's estimate proves that for a twice continuously differentiable function $f:\R \to \R$ defined on some interval $I \subset \R$ where $f''$ is bounded away from zero and any $\epsilon>0$ for all sufficiently large $Q \in \N$ one has that
\begin{equation} \label{hux}
\#\left\{\left( \frac{p_{1}}{q},\frac{p_{2}}{q} \right) \in \Q^{2}: \begin{array}{c}
q \leq Q, \, \, \frac{p_{1}}{q} \in I, \\
\left| f \left( \frac{p_{1}}{q} \right)- \frac{p_{2}}{q} \right| < \psi(Q)/Q
\end{array} \right\} \leq \psi(Q)Q^{2+ \epsilon},
\end{equation}
provided $q \psi(q) \to \infty$ as $q \to \infty$. Later Vaughan and Velani \cite{VV06} replaced the $Q^\epsilon$ term from Huxley's estimate by a constant, thus obtaining the best possible (up to that constant) estimate. Also note that a matching lower bound was obtained in \cite{BDV07} for $C^3$ functions $f$ and this was later extended to a class of $C^1$ functions \cite{BZ10}. Furthermore, \cite{BDV07} and \cite{BZ10} establish the ubiquity property of rational points near the planar curves in question, which lies at the heart of the proof of the lower bound in Theorem~\ref{euclid:curves}.

However, as was later discovered in \cite{BLVV17} in the case of equal weights, the lower bound in Theorem~\ref{euclid:curves} and indeed for $C^2$ manifolds in arbitrary dimensions can be proven using the Mass Transference Principle (MTP) of \cite{BV06} that has become a standard part of the toolkit when attacking such problems. In fact, in this paper we will too utilise the MTP, or rather its more recent versions. Regarding approximations with weights, the following theorem was established in \cite{BLW20}.

\begin{theorem} [See \cite{BLW20}]\label{BLW}
Let $\m := \left\{ (\bx, f(\bx)) : \bx \in \U \subset \R^{d} \right\}$ where $f: \U \rightarrow \R^{m}$ with $ f \in C^{(2)}$. Let $\btau = (\tau_{1}, \dots, \tau_{n}) \in \R^{n}_{+}$ with
\begin{equation*}
\tau_{1} \geq \dots \geq \tau_{d} \geq \max_{d+1 \leq i \leq n} \left\{ \tau_{d+1}, \frac{1-\sum_{i=1}^{m}\tau_{d+i}}{d} \right\} \quad \text{and} \quad \sum_{i=1}^{m}\tau_{d+i}<1.
\end{equation*}
Then
\begin{equation*}
\dim \left( \W_{n}(\btau) \cap \m \right) \geq \underset{1 \leq j \leq d}{\min} \left\{ \frac{n+1 + \sum_{i=j}^{n}(\tau_{j}-\tau_{i})}{\tau_{j}+1} - m \right\}.
\end{equation*}
\end{theorem}

In dimensions $n>2$ the complementary upper bound is known only in the case of equal exponents $\tau_1=\dots=\tau_n=\tau$. Furthermore, there are various constrains on the manifolds and the exponent $\tau$, see \cite{BVVZ17}, \cite{HL18}, \cite{H20}, \cite{S18} and \cite{BY21}. The difficulty in obtaining the upper bounds is primarily associated with the difficulty to count rational points lying close to a manifold.

Note that unlike Theorems~\ref{euclid:curves} and \ref{BLW} and other results for manifolds, Rynne's theorem does not have any constrains on $\btau$ except the absolutely necessary requirement that $\sum_{i=1}^{n} \tau_{i} \geq 1$. In particular, the exponents of approximation in Theorems~\ref{euclid:curves} and \ref{BLW} have to obey certain upper and lower bounds. The case of larger exponents has also been investigated, albeit it relies on studying rational points on manifolds. For example, \cite{MR2604984} computes the Hausdorff dimension of $\W_n(\tau)\cap C$ for polynomial curves $C$ defined over $\Q$ for $\tau\ge \max(\deg(C)-1,1)$.
Further generalisations of \cite{MR2604984} can be found in \cite{MR4089038,MR3450571,MR3708521} where the condition on $\tau$ was relaxed.

The main goal of this paper is to kick-start a similar theory in the case of Diophantine approximations to $p$-adic variables. Specifically, we establish $p$-adic analogues of Theorems~\ref{khintch}, ~\ref{rynne}, and ~\ref{BLW}.

\section{$p$-adic approximations: overview and new results} \label{p-adic}

As mentioned above, this paper is primarily concerned with establishing the $p$-adic versions of the results mentioned in \S\ref{Intro}. For a general introduction  to the theory of $p$-adic numbers and their functions we refer the reader to \cite{G20, M73, S06}. Throughout the rest of this paper $ p \in \N$ will be a fixed prime number and
$\Q_p$ will stand for the completion of $\Q$ with respect to the $p$-adic absolute value
$| \cdot |_p : \Q \to [0, \infty)$, where
\[
	| x |_p = \begin{cases}
				p^{- \nu_p (x)} & \text{ if } x \not= 0, \\
				0 & \text{ if } x = 0,
		        \end{cases}
\]
and
$ \nu_p( x)$ is the unique integer $\ell$ such that $ x = p^{\ell}(a/b) $, $ a, b \in \Z$ and $ \gcd(a, p) = \gcd(b,p) = 1$. Let $\Z_p:=\{x\in\Q_p:|x|_p\le1\}$ be the set of $p$-adic integers and let $\mu_p$ denote the (uniquely defined translation invariant) Haar measure on $\Q_p$ normalised so that $\mu_p( \Zp) = 1$. When considering the space $\Q_p^n$, we will denote the corresponding product measure by $\mu_{p, n}$. Thus, $\mu_{p,n}(\Zp^n)=1$.

The real setting described in \S\ref{Intro} moves readily enough into the $p$-adic setting.
Let $\Psi=(\psi_1,\dots,\psi_n)$ be as in \S\ref{Intro} and for any $a_{0} \in \N$ let
\begin{equation}\label{Ap}
\Ap_{a_{0}}(\Psi)= \bigcup_{\substack{(a_1,\dots,a_n) \in\Z^n\\ |a_{i}| \leq a_{0}\; (1\le i\le n)}} \left\{ \bx=(x_1,\dots,x_n) \in \Zp^{n}: \left| x_{i} - \frac{a_{i}}{a_{0}} \right|_{p} < \psi_{i}(a_{0})\;\;\text{for all } 1 \leq i \leq n \right\}\,.
\end{equation}
Define the set of $p$-adic simultaneously $\Psi$-approximable points in $\Z_p$ as
\begin{equation*}
\Wp_{n}(\Psi)= \limsup_{a_{0} \to \infty} \Ap_{a_{0}}(\Psi).
\end{equation*}
Similarly to the real case we adopt the following simplified notation for $\Wp_{n}(\Psi)$ for $\Psi$ of a special form: $\Wp_{n}(\psi)$ if $\psi_{1}= \dots= \psi_{n}=\psi$; $\Wp_{n}(\btau)$ if $\psi_{i}(q)=q^{-\tau_{i}}$ for some $\btau=(\tau_{1}, \dots \tau_{n}) \in \R^{n}_{+}$; and
$\Wp_{n}(\tau)$  if furthermore $\btau=(\tau,\dots,\tau)$ for some $\tau>0$.

It will be convenient to consider the slightly smaller subset of $\Wp_{n}(\Psi)$ defined by requiring that the rational approximations in each coordinate are reduced rational fractions. Indeed, this is the setting that was considered by Haynes in \cite{H10}, where he showed that, in the case $\psi$ is not monotonic, establishing a zero-one law for $\Sp_{n}(\psi)$ requires this additional condition. For each $a_0\in\N$, let
\begin{equation}\label{Ap'}
\Ap_{a_{0}}'(\Psi)= \bigcup_{\substack{(a_1,\dots,a_n) \in\Z^n\\ |a_{i}| \leq a_{0}\;\&\;(a_{i},a_0)=1 \; (1\le i\le n)}} \left\{ \bx \in \Zp^{n}: \left| x_{i} - \frac{a_{i}}{a_{0}} \right|_{p} < \psi_{i}(a_{0}) \;\;\text{for all } 1 \leq i \leq n \right\}
\end{equation}
and define the corresponding limsup set as
\begin{equation*}
\Wp_{n}'(\Psi)= \limsup_{a_{0} \to \infty} \Ap'_{a_{0}}(\Psi).
\end{equation*}

We now state the main results of this paper, which are the $p$-adic analogues of Theorems~\ref{khintch}, ~\ref{rynne}, and ~\ref{BLW}. We begin with the $p$-adic equivalent of Theorem~\ref{khintch}.

%
%

\begin{theorem} \label{wkpadic}
Let $\psi_{i}: \N \to \R_+$ be approximation functions with $\psi_{i}(q)\ll \frac{1}{q}$ for each $1 \leq i \leq n$ and let $\Psi=(\psi_{1}, \dots, \psi_{n})$. Suppose that $\prod_{i=1}^{n} \psi_{i}$ is monotonically decreasing. Then
\begin{equation*}
\mu_{p,n}(\Wp_{n}'(\Psi))= \begin{cases}
0 \quad \text{ if } \, \, \sum_{q=1}^{\infty} q^{n} \prod_{i=1}^{n} \psi_{i}(q) < \infty, \\
1 \quad \text{ if } \, \, \sum_{q=1}^{\infty} q^{n} \prod_{i=1}^{n} \psi_{i}(q) = \infty.
\end{cases}
\end{equation*}
\end{theorem}
\begin{remark}\label{rem2.2} \rm
Note that the condition that each $\psi_{i}(q) \ll \frac{1}{q}$ is necessary, since the $p$-adic distance between any two rational integers can be made arbitrarily small. This is in stark contrast to the real case where $\psi(q)<\frac{1}{2}$ is sufficient to ensure rectangles in the same 'layer' $\A_{q}(\Psi)$ are non-intersecting.
\end{remark}
The $p$-adic simultaneous version of this theorem, that is when $\psi_1 = \dots = \psi_n$, was established by Jarn\'ik in \cite{J45}. Jarn\'ik's theorem was further generalised by Lutz \cite{L55} to systems of linear forms. We remark that the monotonicity condition is only required in the divergence case. We remove the monotonicity condition of Theorem~\ref{wkpadic} by establishing the following Duffin-Schaeffer type theorem \cite{DS41}.

\begin{theorem}\label{DS_wkpadic}
Let $\psi_i: \N \to \R_+$ be approximation functions with $\psi_{i}(q) \ll \frac{1}{q}$ for $1 \leq i \leq n$ and let $\Psi=(\psi_1, \dots , \psi_n)$. Suppose that
\begin{equation} \label{DS_condition}
\limsup_{N \to \infty} \frac{\sum_{q=1}^{N} \varphi(q)^{n} \prod_{i=1}^{n}\psi_{i}(q)}{\sum_{q=1}^{N} q^{n}\prod_{i=1}^{n}\psi_{i}(q)}>0\,,
\end{equation}
where $\varphi$ is the Euler phi-function. Then
\begin{equation*}
\mu_{p,n}(\Wp_{n}'(\Psi))= \begin{cases}
0 \quad \text{ if } \, \, \sum_{q=1}^{\infty} \varphi(q)^{n} \prod_{i=1}^{n} \psi_{i}(q) < \infty, \\
1 \quad \text{ if } \, \, \sum_{q=1}^{\infty} \varphi(q)^{n} \prod_{i=1}^{n} \psi_{i}(q) = \infty.
\end{cases}
\end{equation*}
\end{theorem}

Note that in both Theorem~\ref{wkpadic} and Theorem~\ref{DS_wkpadic} if $\prod_{i=1}^{n} \psi_{i}(q)< q^{-n-1-\epsilon}$ for any $\epsilon>0$, then $\mu_{p,n}(\Wp_{n}'(\Psi))=0$. To quantify the size of this set further depending on how small this product is we use Hausdorff dimension to establish a $p$-adic version of Theorem~\ref{rynne}.

\begin{theorem} \label{weighted_jb}
 Let $\btau=(\tau_{1}, \dots, \tau_{n}) \in \R^{n}_{+}$ be such that $\sum_{i=1}^{n}\tau_{i} > n+1$ and $\tau_{i}>1$ for each $1 \leq i \leq n$. Then
 \begin{equation*}
 \dim \Wp_{n}(\btau) = \min_{1 \leq i \leq n} \left\{ \frac{n+1+\sum_{\tau_{j}<\tau_{i}}(\tau_{i}-\tau_{j})}{\tau_{i}} \right\}.
 \end{equation*}
 \end{theorem}

 \begin{remark}\rm
 We note that the condition on the summation of the exponent vector $\btau$ is present due to the fact that if $\sum_{i=1}^{n}\tau_{i} \leq n+1$, then, by the $p$-adic version of Dirichlet's Theorem, we have that $\Wp_{n}(\btau)=\Zp^{n}$.
 \end{remark}

 \begin{remark} \rm
 The condition that each $\tau_{i}>1$ may seem unnecessarily restrictive, however, the following reasoning shows why this must be the case. Similarly to Remark~\ref{rem2.2} the key reasoning behind the condition is that $\Z$ is dense in $\Zp$ so in any coordinate axis where $\tau_{i}<1$ all points along the axis can be approximated, regardless of the choice of $a_{0}$ in our approximation sets. If, for example, we considered the approximation set $\Wp_{2}((1-\varepsilon, \tau_{2}))$ for $\varepsilon>0$ and $\tau_{2}>2$ then the above argument gives us that $\Wp_{2}((1-\varepsilon,\tau_{2}))=\Zp \times \Wp_{1}(\tau_{2})$. Using well known bounds on the Hausdorff dimension of product spaces (see e.g \cite{T82}) we have that
 \begin{equation*}
 \dim \Wp_{1}(\tau_{2}) + \dim \Zp \leq \dim \Wp_{2}((1-\varepsilon, \tau_{2})) \leq \dim \Wp_{1}(\tau_{2}) +\dim_{B}\Zp,
 \end{equation*}
 where $\dim_{B}$ is the box-counting dimension, we have that
 \begin{equation}\label{eqn-6}
 \dim \Wp_{2}((1-\varepsilon, \tau_{2}))=\frac{2}{\tau_{2}}+1.
 \end{equation}
 However, if Theorem \ref{weighted_jb} was applicable we would have that
 \begin{equation*}
 \dim \Wp_{2}((1-\varepsilon, \tau_{2}))=\min \left\{ \frac{3+(\tau_{2}-(1-\varepsilon))}{\tau_{2}}, \frac{3}{1-\varepsilon} \right\}=\frac{2}{\tau_{2}}+\frac{\tau_{2}+\varepsilon}{\tau_{2}}\,,
 \end{equation*}
 contrary to \eqref{eqn-6}.
 \end{remark}

Theorem~\ref{weighted_jb} can be further extended to general approximation functions. Suppose that the limits
 \begin{equation} \label{general}
 v_{i}= \lim_{q \to \infty} \frac{-\log \psi_{i}(q)}{\log q},
 \end{equation}
exist and are positive for each $1 \leq i \leq n$. Define the exponents vector $\mathbf{v}=(v_{1}, \dots, v_{n}) \in \R^{n}_{+}$.

 \begin{corollary}
 Let $\Psi$ be such that the limits \eqref{general} exist and are positive. Suppose that $\sum_{i=1}^{n} v_{i}>n+1$ and each $v_{i}>1$. Then
 \begin{equation*}
 \dim \Wp_{n}(\Psi) = \min_{1 \leq i \leq n} \left\{ \frac{n+1+\sum_{v_{j}<v_{i}}(v_{i}-v_{j})}{v_{i}} \right\}.
 \end{equation*}
 \end{corollary}
 \begin{proof}
 By the condition that each function $\psi_{i}$ has corresponding positive limit \eqref{general}, for any $\epsilon>0$ we have that
 \begin{equation*}
 q^{-(v_{i}+\epsilon)} \leq \psi_i(q) \leq q^{-(v_{i}-\epsilon)}  \quad (1 \leq i \leq n)
 \end{equation*}
 for all sufficiently large $q \in \N$. Let $\bm\epsilon=(\epsilon, \dots, \epsilon) \in \R^{n}_{+}$. Then, we have that
 \begin{equation*}
 \Wp_{n}(\mathbf{v}+ \bm\epsilon ) \subseteq \Wp_{n}(\Psi) \subseteq \Wp_{n}(\mathbf{v}- \bm\epsilon).
 \end{equation*}
 By letting $\epsilon \to 0$, and applying Theorem~\ref{weighted_jb} we get the required result.
 \end{proof}

When it comes to $p$-adic approximations on curves and manifolds, less is known. In \cite{KT07} Kleinbock and Tomanov generalized the key results from \cite{KM98} to the $S$-arithmetic setting, which includes the $p$-adic setting. In particular, Kleinbock and Tomanov proved that under the natural assumption $\sum_{i=1}^{n}\tau_{i} > n+1$ the set $\Wp_{n}(\btau) \cap C$ has zero measure on $C$ for a large and natural class of manifolds in $\Q_p^n$. Whilst there are no results relating to the Haar measure of $\Wp_{n}(\Psi) \cap C$ for $C$ a $p$-adic curve or manifold in the case $\Psi$ is a general $n$-tuple of approximation functions, there are several results for dual approximation including inhomogeneous setting, see \cite{BBK05, BK03, MR2919847, MR2863836, MR2670212, DG1, DG2, MG12}.
Regarding the Hausdorff dimension of $\Wp_{n}(\btau) \cap C$, Bugeaud, Budarina, Dickinson, and O'Donnell \cite{MR2837469} and more lately Badziahin, Bugeaud and Schleischitz \cite{BBS20} calculated $\dim(\Wp_{n}(\tau) \cap C)$ in the case $C=(x,\dots,x^n)$ for relatively large exponents $\tau$. Apart from these pair of findings nothing else seems to be known. In this paper we obtain a sharp lower bound on the dimension of $\Wp_{n}(\btau) \cap C$ for a natural and very general class of manifolds defined over $\Z_p^{d}$ and relatively small exponent vector $\btau$. Specifically we will consider manifolds immersed by maps with the following property, which is a multivariable analogue of $C^{1}$ functions given in for example \cite{S11}.

\begin{definition}\rm \label{DQE}
A function $f: \U \to \Qp$ defined on an open set $\U\subset\Qp^{d}$ will be referred to as {\em differentiable with quadratic error} (\DQE) at $\bx\in\U$ if there exists constants $C>0$ and $\varepsilon>0$ and $p$-adic numbers $\partial f(\bx)/\partial x_{\ell}\in\Q_p$ $(1\le \ell\le d)$, which will be referred to as partial derivatives of $f$ at $\bx$, such that for any $\by \in B(\bx,\varepsilon)\subset\U$
\begin{equation}\label{eqn007}
\left|f(\by)-f(\bx)-\sum_{i=1}^{d}\frac{\partial f(\bx)}{\partial x_{i}}(y_i-x_i)\right|_{p}< C\max_{1 \leq i \leq d} |y_{i}-x_{i}|_{p}^{2}\,.
\end{equation}
We will say that a map $\bff=(f_1,\dots,f_m):\U\to\Z_p^m$ is \DQE{} at $\bx$ if each coordinate function $f_j$ is \DQE{} at $\bx$. We will say that $f$ (resp. $\bff$) is \DQE{} on $\U$ if it is \DQE{} at each point $\bx\in\U$.
\end{definition}

\begin{remark}\rm
Note that if the right hand side of \eqref{eqn007} was simply $o\left(\max_{1 \leq j \leq d} |y_{j}-x_{j}|_{p}\right)$ then $f$ would be simply differentiable at $\bx$. The above definition imposes a stronger condition than differentiability in the sense that the error term in \eqref{eqn007} is quadratic. It is readily verified that any $C^2$ function, as defined in \cite{S06} (see also \cite{KT07} for a brief survey of $p$-adic $C^k$ functions), is \DQE{} at every point. The converse may not be true. Mahler's normal functions are $C^{\infty}$ and so they are \DQE.
\end{remark}

%
%

We are now in  position to state our results for $\btau$-approximable points on manifolds given as $\Cf=\{(\bx,\bff(\bx)):\bx\in\U\}$ for some map $\bff:\U\to\Q_p^m$, $\U\subset\Z_p^d$, thus extending Theorem~\ref{BLW} to the $p$-adic setting. This can be done in two ways: by stating our results for exactly the set $\Wp_n(\btau) \cap \Cf$, or by stating them for the set of $\bx\in\U$ such that $\ff(\bx):=(\bx,\bff(\bx))\in \Wp_n(\btau)$. We opt for the latter since it requires fewer assumptions, albeit the two ways are equivalent if we assume that $\bff$ is a Lipschitz map, which follows from Proposition~\ref{falc}. Thus, our statements will be about the Hausdorff measure and dimension of
$$
\ff^{-1}\left(\Wp_n(\btau)\right):=\{\bx\in\U:\ff(\bx)=(\bx,\bff(\bx))\in\Wp_n(\btau)\}\,.
$$
It is easily seen that this set is a subset of the projection of $\Wp_n(\btau)$ onto the first $d$ coordinates.

\begin{theorem} \label{lowerbnd}
Let $\bff: \U \to \Zp^{m}$ be \DQE{} on an open set $\U\subseteq \Zp^{d}$ and for $\bx\in\U$ let $\ff(\bx)=(\bx,\bff(\bx))$. Suppose
\begin{equation*}
1+\frac{1}{n} < \tau < 1+ \frac{1}{m},
\end{equation*}
Then
\begin{equation}\label{vb1}
\dim (\ff^{-1}\left(\Wp_n(\tau)\right) \geq s:= \frac{n+1}{\tau}-m\,.
\end{equation}
Furthermore, if $\bff$ is Lipschitz on $\U$ then for any ball $B\subset\U$
\begin{equation}\label{vb1+}
\ha^{s}\left(B\cap \ff^{-1}\left(\Wp_n(\tau)\right)\right)= \ha^{s}(B).
\end{equation}
\end{theorem}

\begin{theorem} \label{lowerbnd1}
Let $\bff$, $\U$ and $\ff$ be as in Theorem~\ref{lowerbnd} and additionally assume that $d=1$ and so $m=n-1$.
Suppose that $\btau=(\tau_{1},\tau_{2}, \dots, \tau_{n})\in \R^{n}_{+}$ satisfies the conditions
\begin{equation*}
\tilde\tau:=\sum_{j=2}^{n} \tau_{j} < n, \quad \quad \tau_{1} \geq \max_{2 \leq i \leq n}\left\{ \tau_{i}, n+1-\tilde\tau \right\} \quad \text{and} \quad \tau_{i}>1 \, \, \,  \text{for } \, \, 2 \leq i \leq n.
\end{equation*}
 Then
\begin{equation}
	\label{eqn:lwrbnd}
\dim \ff^{-1}\left(\Wp_n(\btau)\right) \geq s:= \frac{n+1+\sum_{j=2}^{n}(\tau_{1}-\tau_{j})}{\tau_{1}}-(n-1)=\frac{n+1-\tilde\tau}{\tau_1}.
\end{equation}
Furthermore, if $\bff$ is Lipschitz on $\U$ then for any ball $B\subset\U$
\begin{equation}\label{vb1+3}
\ha^{s}\left(B\cap \ff^{-1}\left(\Wp_n(\btau)\right)\right)=\ha^{s}(B).
\end{equation}
\end{theorem}

\begin{theorem} \label{lowerbnd_manifold}
Let $\bff$, $\U$ and $\ff$ be as in Theorem~\ref{lowerbnd} and suppose that $\btau=(\tau_{1},\tau_{2}, \dots, \tau_{n})\in \R^{n}_{+}$ satisfies the conditions
\begin{equation*}
\tau_{i}>1 \quad (1 \leq i \leq n), \quad \quad
\sum_{i=1}^{m}\tau_{d+i}<m+1, \quad \quad \sum_{i=1}^{n} \tau_{i}>n+1, \, \, \text{ and }
\min_{1 \leq i \leq d} \tau_{i} \geq \max_{1 \leq j \leq m}\tau_{d+j}.
\end{equation*}
Then
\begin{equation}\label{vb1+4}
\dim \left(\ff^{-1}\left(\Wp_n(\btau)\right)\right) \geq \min_{1 \leq i \leq d}\left\{ \frac{n+1+\sum_{\tau_{j} < \tau_{i}}(\tau_{i}-\tau_{j})}{\tau_{i}}-m \right\}.
\end{equation}
\end{theorem}

\begin{remark} \rm
Note that the dimension results of Theorems~\ref{lowerbnd}--\ref{lowerbnd1} are contained within Theorem~\ref{lowerbnd_manifold}. However, due to the method of proof we are not able to obtain the Hausdorff measure result in Theorem \ref{lowerbnd_manifold}. Also note that the statements remain true if the assumptions imposed on $\bff$ are imposed on a sufficiently small ball $B\subset\U$ instead of $\U$.
\end{remark}

\begin{remark}\rm
The assumption that the approximations over the independent variables ($\tau_1$ in Theorem \ref{lowerbnd1} and $\tau_{1}, \dots , \tau_{d}$ in Theorem \ref{lowerbnd_manifold}) are larger than the approximations over the dependent variables is merely technical. Observe that this condition is not needed amongst the approximations over each respective variable, since we may permute the variables to obtain the desired ordering.
However, the other requirements placed on $\btau$ are necessary to allow the result to hold for as general set of manifolds as possible. In particular, the conditions that $\sum_{i=1}^{m} \tau_{d+i}<m+1$ and $\tau_{i}>1$ for $1 \leq i \leq n$ ensure that even if the manifold is a hyperplane passing through badly approximable points we will still have an infinite number of rational approximations. If these conditions do not hold a counterexample can be readily obtained on modifying the example of Remark~3 in \cite{BLVV17}. It is also easy to see that the lower bound $\tau_{1} \geq n+1-\tilde\tau$ is necessary for otherwise \eqref{eqn:lwrbnd} would be false.
The upper bound $\tilde\tau<n$ on $\tilde\tau$ can likely be improved, however this will require imposing additional conditions on the curves such as non-degeneracy (meaning $1,x,f_1(x),\dots,f_{n-1}(x)$ are linearly independent over $\Z_p$ in the case these are analytic/Mahler's normal functions), and will require a different approach such as that of \cite{B12}. We plan to address the problem for non-degenerate curves separately in a subsequent publication.
\end{remark}

\begin{remark}\rm
We expect that the lower bound of Theorem~\ref{lowerbnd}-\ref{lowerbnd_manifold} is sharp and each dimension result should indeed be equality at least for non-degenerate curves.
Obtaining the upper bounds represents a challenging open problem even in dimension 2. We would like to stress that there is currently no equivalent to Huxley's estimate \cite{H96} in the $p$-adic setting, let alone the sharper Vaughan-Velani result \cite{VV06}.
\end{remark}

%
%

\section{Auxiliary concepts and results}
\label{sect:aux}

Before giving the proofs of Theorems~\ref{wkpadic}, ~\ref{weighted_jb} and ~\ref{lowerbnd}-\ref{lowerbnd_manifold}, we collect together some auxiliary results and concepts which we will need.
We begin by recalling the Borel-Cantelli lemma \cite{B09}, which can be found in numerous publications and texts on probability theory.

\begin{lemma} \label{borel_cantelli_convergence}
Let $(\Omega, \mathcal{A},  \mu)$, be a measure space and $(E_{i})_{i=1}^\infty$ is a sequence of $\mu$-measurable subsets in $\Omega$ satisfying
\begin{equation*}
\sum_{i=1}^{\infty} \mu(E_{i})< \infty.
\end{equation*}
Then,
\begin{equation*}
\mu \left(\limsup_{i \to \infty}E_{i} \right)=0.
\end{equation*}
\end{lemma}

We will also use the following converse of the Borel-Cantelli lemma, see \cite[p.\,165]{S67} or \cite[Chap.1,\S3]{S79}, which was first established in a slightly weaker form by Erdos and Renyi \cite{ER59} and for arbitrary probability spaces by Kochen and Stone \cite{KS64}.

\begin{lemma}\label{borel_cantelli_divergence}
Let $(\Omega, \mathcal{A},  \mu)$ be a measure space with
$\mu(\Omega) <\infty$.  Suppose that $(E_{i})_{i=1}^\infty$ is a family of $\mu$-measurable subsets in $\Omega$ such that
\begin{equation*}
\sum_{i=1}^{\infty} \mu(E_{i})= \infty.
\end{equation*}
Then,
\begin{equation} \label{quasi_independence}
\mu \left( \limsup_{n \rightarrow \infty} E_{i} \right) \geq \limsup_{n \rightarrow \infty} \frac{\Big(\sum_{i=1}^{n} \mu(E_{i}) \Big)^{2}}{\sum_{i,j=1}^{n} \mu\big( E_{i} \cap E_{j}\big)}.
\end{equation}
\end{lemma}

The following lemma is a variant of Minkowski's Theorem for systems of linear forms in the $p$-adic case.

 \begin{lemma} \label{mink}
Let $L_{i}(\bx)$, with $i=1, \dots , n$, be linear forms in $\bx=(x_{0},x_{1}, \dots , x_{n})$ with $p$-adic integer coefficients.  Let $\btau=(\tau_{1},\dots,\tau_n) \in \R_{+}^n$ satisfy $\sum_{i=1}^{n} \tau_{i}=n+1$ and $\bsigma=(\sigma_{1},\dots,\sigma_n) \in \R^n$ satisfy $\sum_{i=1}^{n} \sigma_{i}=n$. Then there exists $H_{\bsigma}>0$ such that for all integers $H_0,\dots,H_n\ge1$ such that $T^{n+1}:=(H_0+1)\cdots (H_n+1)\ge H_{\bsigma}$ there exists a non-zero rational integer vector $\bx=(x_{0},x_{1}, \dots , x_{n})$ satisfying
\begin{equation} \label{hbound}
|x_{i}| \leq H_i\qquad\text{for all }0 \leq i \leq n
\end{equation}
and
\begin{equation}
	\label{eqn:sys}
|L_{i}(\bx)|_{p} \le p^{\sigma_i}T^{-\tau_{i}} \;\;\; \text{ for all } 1\le i\le n\,.
\end{equation}
\end{lemma}

\begin{proof}
The proof of this lemma is standard and uses Dirichlet's pigeonhole principle. Nevertheless, as it is used in the proofs of our main results, for completeness we provide its details here. To begin with, note that there are $T^{n+1}$ different rational integer vectors $ \bx=(x_{0}, \dots , x_{n})$ satisfying \eqref{hbound}, subject to the condition that $ x_i \geq 0 $ for each $i$. Let $\varepsilon\in(0,1)$ and $T_{\varepsilon}=T-\varepsilon$.
For each $i=1, \dots, n$ let $\delta_{i}$ be the unique integer such that
\begin{equation}\label{eqn-11}
p^{\delta_{i}-1} \leq p^{-\sigma_i}T_{\varepsilon}^{\tau_{i}} < p^{\delta_{i}}\,.
\end{equation}
Assuming $H_{\bsigma}$, which can be found explicitly, is sufficiently large we ensure that $\delta_i\ge0$ for each $i$. Clearly, for each $\bx\in\Z^n$ we have that $L(\bx):=(L_1(\bx),\dots,L_n(\bx))\in\Z_p^n$. Split $\Z_p^n$ into the subsets $S(\ba)$ given by
$$
S(\ba)=\prod_{i=1}^n \{x_i\in\Z_p:|x_i-a_i|_p\le p^{-\delta_i}\}
$$
for each $\ba=(a_1,\dots,a_n)\in\Z^n$ with $0\le a_i< p^{\delta_i}$. It is readily seen that the sets $S(\ba)$ are disjoint and cover the whole of $\Z_p^n$.
Furthermore, using the facts that $\sum_i\tau_i=n+1$, $\sum_i\sigma_i=n$ and \eqref{eqn-11}, we find that the number of sets $S(\ba)$ is
$$
p^{\sum_i\delta_i}\le T_{\varepsilon}^{\sum_i\tau_i}=T_{\varepsilon}^{n+1}<T^{n+1}\,.
$$
Hence, by the pigeonhole principle, at least one of the sets $S(\ba)$ contains  $L(\bx_i)$ for at least two distinct integer points $\bx_1$ and $\bx_2$ as specified above. Let $\bx=\bx_1-\bx_2$. Clearly, \eqref{hbound} is satisfied and $\bx$ is non-zero. Furthermore, for each $i=1,\dots,n$ we have that
 \begin{equation}\label{eqn011}
 |L_i(\bx)|=|L_{i}(\bx_{1}-\bx_{2})|_{p} = |L_{i}(\bx_{1})-L_i(\bx_{2})|_{p} \leq p^{-\delta_{i}} \stackrel{\eqref{eqn-11}}{<} p^{\sigma_i}T_{\varepsilon}^{-\tau_{i}}\,.
 \end{equation}
Since there are only finitely many integer vectors $ \bx=(x_{0}, \dots , x_{n})$ satisfying \eqref{hbound}, there is a non-zero $\bx$ subject to \eqref{hbound} satisfying \eqref{eqn011} for every $\varepsilon\in(0,1)$. Letting $\varepsilon\to0$ verifies \eqref{eqn:sys} and completes the proof.
\end{proof}

It is well known that the Hausdorff dimension is preserved by bi-Lipschitz mappings. We now state this formally in relation to $\Z_p^n$ for future reference, but omit the proof as it is a very well known fact. For instance, in the Euclidean case this can be found in \cite{F14}. In what follows, given a vector $\bx=(x_{1}, \dots, x_{n})\in\Z_p^n$, we define $|\bx|_{p}:=\max \big\{ |x_{i}|_{p}:1\le i\le n \big\}$.

\begin{proposition} \label{falc}
Let $F \subset \Zp^{n}$ and $g: F \rightarrow \Zp^{m}$ be a Lipschitz map, that is
$
|g(x)-g(y)|_{p} \leq c |x-y|_{p} \, \text{ for } \, x,y \in F
$
for some constant $c>0$. Then, for each $s>0$
\begin{equation*}
\ha^{s}(g(F)) \leq c^{s} \ha^{s}(F)
\end{equation*}
and so $\dim g(F)\le\dim F$. In particular, if $F$ is a bi-Lipschitz map, then
$
\dim g(F) = \dim F.
$
\end{proposition}


\subsection{Mass Transference Principles}

For the proofs of Theorems~\ref{weighted_jb} and ~\ref{lowerbnd} we will use the Mass Transference Principle (MTP) of \cite{BV06} and some of its recent generalisations. The development of the MTP in \cite{BV06} has prompted a significant amount of subsequent research  and is now part of the standard machinery for studying many problems in metric Diophantine approximation, see \cite{AT19} for a survey. The first generalisation of the MTP was for systems of linear forms established in \cite{BV06b}, which was further generalised in \cite{AB18}. Subsequently, Allen and Baker \cite{AB19} proved a general MTP for a wide variety of sets satisfying certain conditions, these sets included self similar sets and smooth compact manifolds. A generalisation of the MTP capable of dealing with problems on Diophantine approximation with weights was established by Wang, Wu and Xu \cite{WWX15}. More recently, Wang and Wu \cite{WW19} established a stronger and in a sense more versatile version of the MTP obtained in \cite{WWX15}, albeit this stronger version requires a ubiquity hypothesis similar to that of \cite{BDV06}. In this paper we will deploy this latest result of Wang and Wu to establish Theorem~\ref{weighted_jb}, in which context verifying the ubiquity hypothesis is relatively simple. Regarding Theorem~\ref{lowerbnd} we will utilise the original MTP of \cite{BV06}. We expect that on blending the techniques of this paper together with other variations of the MTP our ideas can be carried forwards to establish results similar to Theorem~\ref{weighted_jb} for systems of linear forms and Theorem~\ref{lowerbnd} to manifolds of arbitrary dimension. However, attaining optimal conditions on the exponents will likely require additional considerations.

In what follows, let $(X,d)$ be a locally compact metric space. A continuous function $g:(0,\infty) \to (0, \infty)$ is said to be \textit{doubling} if there exists a constant $\lambda>1$ such that for all $x>0$,
\begin{equation*}
g(2x) \leq \lambda g(x).
\end{equation*}
Suppose there exists constants $0 <c_{1}<1<c_{2}< \infty$ and $r_{0}>0$ such that
\begin{equation} \label{general_MTP_requirement}
c_{1}g(r(B)) \leq \ha^{g}(B) \leq c_{2}g(r(B)),
\end{equation}
for any ball $B=B(x,r)$ centred at $x \in X$ with radius $r(B)=r \leq r_{0}$. Then we will say that $(X, d)$ is \textit{$g$-Ahlfors regular}. Next, given a dimension function $f$ and a ball $B=B(x,r)$, define
\begin{equation}\label{eq15}
B^{f, g} = B\left(x,g^{-1}(f(r))\right).
\end{equation}
Note that the centre and radius of a ball may not be unique. Indeed, this is the case in $\Q_p$.
Thus, in \eqref{eq15} and elsewhere by a ball we understand the pair of its centre and radius. However, when using a ball in set theoretic expressions by a ball we will mean the corresponding set of points.
Also note that, by \eqref{eq15}, we have that $B^{g, g}=B$. Now we are ready to state the general MTP as given in ~\cite[Theorem~3]{BV06}, see also \cite[Theorem~1]{AB19}.

\begin{theorem}[General Mass Transference Principle] \label{general_MTP}
Let $(X,d)$ be a locally compact metric space, $g$ be a doubling dimension function. Suppose that $(X,d)$ is $g$-Ahlfors regular. Let $(B_{i})_{i \in \N}$ be a sequence of balls in $X$ with $r(B_{i}) \rightarrow 0$ as $i \rightarrow \infty$. Let $f$ be a dimension function such that $f(x)/g(x)$ is monotonic and suppose that for any ball $B \subset X$
\begin{equation*}
\ha^{g}\left( B \cap \limsup_{i \rightarrow \infty} B_{i}^{f,g} \right)=\ha^{g}(B).
\end{equation*}
Then, for any ball $B \subset X$
\begin{equation*}
\ha^{f}\left( B \cap \limsup_{i \rightarrow \infty} B_{i} \right)=\ha^{f}(B).
\end{equation*}
\end{theorem}

Now let us turn our attention to the Mass Transference Principle from Rectangles to Rectangles (MTPRR) of Wang and Wu \cite{WW19}, which will be a vital component of our proof of Theorem~\ref{weighted_jb}. To begin with, we state the notion of local ubiquity for rectangles introduced in \cite{WW19}, which is a generalisation of the notion of local ubiquity for balls introduced in \cite{BDV06}. Fix an integer $n \geq 1$, and for each $1 \leq i \leq n$ let $(X_{i},|\cdot|_{i},m_i)$ be a bounded locally compact measure-metric space, where $|\cdot|_i$ denotes the metric and $m_{i}$ denotes a measure over $X_i$, which will be assume to be a $\delta_{i}$-Ahlfors regular probability measure. Consider the product space $(X,|\cdot|,m)$, where
\begin{equation*}
X=\prod_{i=1}^{n}X_{i}, \quad m=\prod_{i=1}^{n}m_{i}, \quad |\cdot|=\max_{1 \leq i \leq n}|\cdot|_{i}
\end{equation*}
are defined in the usual way.
In view of the application we have in mind, we will take $X_{i}=\Zp$, $m_{i}=\mu_{p}$ and $|\cdot|_i = | \cdot |_p$ for each $1 \leq i \leq n$. So $X=\Zp^{n}$, $m=\mu_{p,n}$, and $|\cdot|$ is the usual $\sup$ norm.
For any $x \in X$ and $r \in \R_{+}$ define the open ball
\begin{equation*}
B(x,r)=\left\{ y \in X: \max_{1 \leq i \leq n}|x_{i}-y_{i}|_{i}< r \right\}=\prod_{i=1}^{n}B_{i}(x_{i},r),
\end{equation*}
where $B_{i}$ are the usual open $r$-balls associated with the $i^{\text{th}}$ metric space $X_i$. Let $J$ be a countably infinite index set, and $\beta: J \to \R_{+}$, $\alpha \mapsto \beta_{\alpha}$ a positive function satisfying the condition that for any $N \in \N$
\begin{equation*}
\# \left\{ \alpha \in J: \beta_{\alpha} < N \right\} < \infty.
\end{equation*}
 Let $l_{n},u_{n}$ be two sequences in $\R_{+}$ such that $u_{n} \geq l_{n}$ with $l_{n} \to \infty$ as $n \to \infty$. Define
\begin{equation*}
J_{n}= \{ \alpha \in J: l_{n} \leq \beta_{\alpha} \leq u_{n} \}.
\end{equation*}
Let $\rho: \R_{+} \to \R_{+}$ be a non-increasing function such that $\rho(x) \to 0$ as $x\to \infty$. For each $1 \leq i \leq n$, let $( R_{\alpha,i})_{\alpha \in J}$ be a sequence of subsets in $X_{i}$.
The family of sets $( R_\alpha)_{\alpha\in J}$ where
\begin{equation*}
	R_{\alpha}=\prod_{i=1}^{n} R_{\alpha, i}, 
\end{equation*}
for each $ \alpha \in J$, are called \textit{resonant sets}.
For $\ba=(a_{1}, \dots, a_{n}) \in \R_{+}^{n}$ define
\begin{equation*}
\Delta(R_{\alpha},\rho(r)^{\ba})= \prod_{i=1}^{n}  \Delta'(R_{\alpha,i},\rho(r)^{a_{i}}),
\end{equation*}
where for some set $A\subset X_i$ and $b \in \R_{+}$
\begin{equation*}
\Delta'(A,b)= \bigcup_{a \in A}B(a,b)
\end{equation*}
is the union of balls in $X_i$ of radius $b$ centred at all possible points in $A$.

\begin{definition}[Local ubiquitous system of rectangles]\rm
Call the pair $\big((R_{\alpha})_{\alpha \in J}, \beta\big)$ {\em a local ubiquitous system of rectangles with respect to $(\rho,\ba)$}\/ if there exists a constant $c>0$ such that for any ball $B \subset X$
\begin{equation*}
\limsup_{n \to \infty} m \left( B \cap \bigcup_{\alpha \in J_{n}}\Delta(R_{\alpha}, \rho(u_{n})^{\ba}) \right) \geq c m(B).
\end{equation*}
\end{definition}
The second property needed to state the Wang-Wu theorem is a local scaling property, which was first introduced in \cite{AB19}, and which is a version of the intersection properties of \cite{BDV06}. In our setting the condition will be satisfied for $k=0$ and holds trivially. Nevertheless, we include the condition for the sake of completeness.

\begin{definition}[$k$-scaling property]\rm
Let $0 \leq k <1$ and $1 \leq i \leq n$. The sequence $\{R_{\alpha,i}\}_{\alpha \in J}$ has {\em $k$-scaling property}\/ if for any $\alpha \in J$, any ball $B(x_{i},r) \subset X_{i}$ with centre $x_{i} \in R_{\alpha,i}$, and $0 < \epsilon < r$ then
\begin{equation*}
c_{2}r^{\delta_{i}k}\epsilon^{\delta_{i}(1-k)} \leq m_{i} \left( B(x_{i},r) \cap \Delta(R_{\alpha,i},\epsilon) \right) \leq c_{3} r^{\delta_{i}k}\epsilon^{\delta_{i}(1-k)},
\end{equation*}
for some constants $c_{2},c_{3}>0$.
\end{definition}
Finally, for $\bt=(t_{1}, \dots, t_{n}) \in \R^{n}_{+}$, define
\begin{equation*}
W(\bt)= \limsup_{ \alpha \in J} \Delta(R_{\alpha},\rho(\beta_{\alpha})^{\ba+\bt}).
\end{equation*}
We now state the following theorems due to Wang and Wu \cite{WW19}.

\begin{theorem}[Mass Transference Principle from Rectangles to Rectangles with Ubiquity] \label{MTPRR}
Let $(X,|\cdot|,m)$ be a product space of $n$ bounded locally compact metric spaces $(X_{i},|\cdot|_{i},m_{i})$ with $m_{i}$ a $\delta_{i}$-Ahlfors probability measure, for $1 \leq i \leq n$. Let $(R_{\alpha})_{\alpha \in J}$ be a sequence of subsets contained in $X$ and assume that $((R_{\alpha})_{\alpha \in J}, \beta)$ is a local ubiquitous system of rectangles with respect to $(\rho, \ba)$ for some $\ba=(a_{1}, \dots, a_{n}) \in \R^{n}_{+}$, and that $(R_{\alpha})_{\alpha \in J}$ satisfies the $k$-scaling property. Then, for any $\bt=(t_{1}, \dots, t_{n}) \in \R^{n}_{+}$
\begin{equation*}
\dim W(\bt) \geq \min_{A_{i} \in A} \left\{ \sum_{j \in K_{1}} \delta_{j}+ \sum_{j \in K_{2}} \delta_{j}+ k \sum_{j \in K_{3}} \delta_{j}+(1-k) \frac{\sum_{j \in K_{3}}a_{j}\delta_{j}-\sum_{j \in K_{3}}t_{j}\delta_{j}}{A_{i}} \right\}=s,
\end{equation*}
where $A=\{ a_{i}, a_{i}+t_{i} , 1 \leq i \leq n \}$ and $K_{1},K_{2},K_{3}$ are a partition of $\{1, \dots, n\}$ defined as
\begin{equation*}
 K_{1}=\{ j:a_{j} \geq A_{i}\}, \quad K_{2}=\{j: a_{j}+t_{j} \leq A_{i} \} \backslash K_{1}, \quad K_{3}=\{1, \dots n\} \backslash (K_{1} \cup K_{2}).
 \end{equation*}
 Furthermore, for any ball $B \subset X$
  \begin{equation} \label{MTPRR_measure}
\ha^{s}(B \cap W(\bt))=\ha^{s}(B).
\end{equation}
 \end{theorem}

 \begin{theorem}[Mass Transference Principle from Rectangles to Rectangles without Ubiquity] \label{MTPRR_full_measure} Suppose that each measure $m_{i}$ is $\delta_{i}$-Ahlfors regular and $R_{\alpha,i}$ has $k$-scaling property for each $\alpha \in J$ $(1 \leq i \leq n)$. Suppose
 \begin{equation*}
 m\left( \limsup_{\alpha \in J} \Delta(R_{\alpha}, \rho(\beta_{\alpha})^{\ba}) \right)=m(X).
 \end{equation*}
 Then
 \begin{equation*}
 \dim W(\bt) \geq s\,,
 \end{equation*}
 where $s$ is the same as in Theorem~\ref{MTPRR}.
 \end{theorem}

 \begin{remark}\rm
 Note that the full measure statement of Theorem \ref{MTPRR_full_measure} is far easier to establish than the local ubiquity statement required in Theorem \ref{MTPRR}, however this shortcut comes at the cost of $s$-Hausdorff measure statement.
 \end{remark}

\vspace*{2ex}

\section{A zero-one law}

In what follows we will need a statement showing that, given a sequence of balls, if the radii of the balls are multiplied by some constant, then the Haar measure of the corresponding $\limsup$ set remains unchanged. We establish this lemma in greater generality for arbitrary ultrametric spaces where such a statement may be useful when solving problems of the same ilk, for example, in Diophantine approximation over locally compact fields of positive characteristic.

\begin{lemma}\label{measure_unchanged}
Let $(X,d)$ be a separable ultrametric space and $\mu$ be a Borel regular measure on $X$. Let $(B_i)_{i\in\N}$ be a sequence of balls in $X$ with radii $r_i\to0$ as
$i\to\infty$. Let $(U_i)_{i\in\N}$ be a sequence of $\mu$-measurable
sets such that $U_i\subset B_i$ for all $i$. Assume that for some
$c>0$
\begin{equation}\label{e:005}
     \mu(U_i)\ge c\mu(B_i)\qquad\text{for all }i\,.
\end{equation}
Then the limsup sets
$$
\textstyle
 \mathcal{U}=\limsup\limits_{i\to\infty}U_i:=\bigcap\limits_{j=1}^\infty\ \bigcup\limits_{i\ge
 j}U_i\qquad\text{ and }\qquad
 \mathcal{B}=\limsup\limits_{i\to\infty}B_i:=\bigcap\limits_{j=1}^\infty\ \bigcup\limits_{i\ge
 j}B_i
$$
have the same $\mu$-measure.
\end{lemma}

\noindent The $\R^n$ version of this statement is well known and can be found for example in \cite[Lemma~1]{MR2457266}, which proof is originally due to Cassels and uses Lebesgue's density theorem. Below we give a full proof of Lemma~\ref{measure_unchanged} for completeness. Our proof is built on the ideas of \cite[Lemma~1]{MR2457266} and \cite[Lemma~1 in Part~II, Ch.~1]{S67}.

\begin{proof}
Let $\mathcal{U}_j:=\bigcup_{i\ge j}U_i$ and $\mathcal{D}_j:=\mathcal{B}\setminus\mathcal{U}_j$. Then,
$\mathcal{D}:=\mathcal{B}\setminus\mathcal{U}=\bigcup_{j}\mathcal{D}_j$ and we need to prove that $\mathcal{D}$ has $\mu$-measure zero. Assume the contrary. Then, since every set $\mathcal{D}_j$ is $\mu$-measurable and $\mathcal{D}_j\subseteq \mathcal{D}_{j+1}$ for all $j$, by the continuity of $\mu$, there is an $\ell\in\N$ such that $\mu(D_\ell)>0$. Since $\mu$ is Borel regular $\mu(\mathcal{D_\ell})=\inf\{\mu(A):\mathcal{D_\ell}\subset A,\;A\text{ is open}\}$. Since $X$ is separable and ultrametric, every open set $A$ can be written as a disjoint countable union of balls. Hence, for any $\varepsilon>0$ there exists a countable collection of disjoint balls $(A_i)$ such that
\begin{equation}\label{vb0}
\mathcal{D}_\ell\subset \bigcup_i A_i\quad\text{and}\quad \sum_i\mu(A_i)-\varepsilon\le \mu(\mathcal{D_\ell})\le \sum_i\mu(A_i)\,.
\end{equation}
Let
$$
\lambda:=\sup\left\{\frac{\mu(A_i\cap\mathcal{D}_\ell)}{\mu(A_i)}:i\in\N,\;\mu(A_i)>0\right\}\,.
$$
Note that, since $\mu(\mathcal{D}_\ell)>0$, the above set is non-empty and therefore $\lambda\in[0,1]$. Then, by \eqref{vb0}, we have that
$$
\mu(\mathcal{D}_\ell)=\sum_i\mu(A_i\cap \mathcal{D}_\ell)\le \lambda\sum_i\mu(A_i)\le \lambda(\mu(\mathcal{D}_\ell)+\varepsilon)\,.
$$
Therefore,
$$
\lambda\ge\frac{\mu(\mathcal{D}_\ell)}{\mu(\mathcal{D}_\ell)+\varepsilon}\,.
$$
Since $\mu(\mathcal{D}_\ell)>0$, on taking $\varepsilon>0$ small enough, we can ensure that $\lambda>1-c$. Then, by the definition of $\lambda$, there exists $i_0\in\N$ such that $\mu(A_{i_0})>0$ and
\begin{equation}\label{vb2}
\frac{\mu(A_{i_0}\cap\mathcal{D}_\ell)}{\mu(A_{i_0})}>1-c\,.
\end{equation}
Take $j\ge \ell$ sufficiently large so that for every $i\ge j$ the radius of $B_i$ is less than the radius of $A_{i_0}$. Then, since $X$ is ultrametric, for all $i\ge j$ if $B_i\cap A_{i_0}\neq\varnothing$ then $B_i\subset A_{i_0}$. Since $\mathcal{D}_\ell\subset \mathcal{D}\subset\mathcal{B}\subset \bigcup_{i\ge j}B_i$, we have that
\begin{equation}\label{eq10}
  A_{i_0}\cap\mathcal{D}_\ell\subset \bigcup_{i\ge j,\; B_i\cap A_{i_0}\neq\varnothing}B_i\cap\mathcal{D}_\ell\,.
\end{equation}
Without loss of generality assume the $B_{i}$ over $i \geq j$ are disjoint, since if not we can take a disjoint sub-collection of $(B_{i})_{i\ge j}$ such that the union of balls in this subcollection is again $\bigcup_{i\ge j}B_i$ and so the sub-collection would satisfy \eqref{eq10}. Such sub-collection is possible to choose since $X$ is ultrametric. Therefore, by \eqref{eq10}, we have that
\begin{equation}\label{vb3}
  \mu(A_{i_0}\cap\mathcal{D}_\ell)\le \sum_{i\ge j,\; B_i\cap A_{i_0}\neq\varnothing}\mu(B_i\cap\mathcal{D}_\ell)\,.
\end{equation}
By construction $\mathcal{D}_{i}\cap
U_{i}=\emptyset$ for every $i$.  Thus, in view of (\ref{e:005}) and the fact that $U_i\subset B_i$ we have that
$$
\mu(B_{i})\ge\mu(U_{i}\cap B_i)+\mu(\mathcal{D}_{i}\cap B_{i})\ge
c\mu(B_{i})+\mu(\mathcal{D}_{i}\cap B_{i})
$$
and so $\mu(\mathcal{D}_{i}\cap B_{i})\le(1-c) \ \mu(B_{i}) $ for all
$i$.  In particular, since $\mathcal{D}_{i}\subset \mathcal{D}_{i+1}$ for all $i$ and $j\ge \ell $ we get that
$$
\mu(\mathcal{D}_{\ell}\cap B_{i})\le \mu(\mathcal{D}_{i}\cap B_{i})\le(1-c) \ \mu(B_{i}) \quad\text{for all }i\ge j\,.
$$
Hence, by \eqref{vb3} and the assumption that the $B_i$ for $i\ge j$ are disjoint, we get that
$$
  \mu(A_{i_0}\cap\mathcal{D}_\ell)\le \sum_{i\ge j,\; B_i\cap A_{i_0}\neq\varnothing}(1-c)\mu(B_i)
 =(1-c)\mu\left(\bigcup_{i\ge j,\; B_i\cap A_{i_0}\neq\varnothing}B_i\right)\le(1-c)\mu(A_{i_0})\,.
$$
This contradicts \eqref{vb2}. The proof is thus complete.
\end{proof}

Note that Lemma~\ref{measure_unchanged} is only applicable to limsup sets contained between two balls with radius varying by some constant. Since many of our sets of interest are $\limsup$ sets of rectangles we make the following extension to Lemma~\ref{measure_unchanged}.

\begin{lemma}\label{measure_unchanged_rectangles}
Let $n\in\N$. For each $1\le j\le n$ let $(X_j,d_j)$ be a separable ultrametric space equipped with a Borel regular $\sigma$-finite measure $\mu_{j}$, $(B^{(j)}_i)_{i\in\N}$ be a sequence of balls in $X_j$ with radii $r^{(j)}_i\to0$ as
$i\to\infty$, $(U^{(j)}_i)_{i\in\N}$ be a sequence of $\mu_j$-measurable sets such that $U^{(j)}_i\subset B^{(j)}_i$ for all $i$ and assume that for some
$c^{(j)}>0$
\begin{equation}\label{e:005z}
     \mu_{j}\left(U^{(j)}_i\right)\ge c^{(j)}\mu_{j}\left(B^{(j)}_i\right) \qquad\text{for all }i\in\N.
\end{equation}
Let $X=\prod_{j=1}^{n}X_{j}$, $d=\max_{j}d_{j}$ be the metric on $X$, $\mu=\prod_{j=1}^n\mu_j$ be the product of measure on $X$ and for each $i\in\N$ let $B_i=\prod_{j=1}^{n}B^{(j)}_i$ and $U_i=\prod_{j=1}^{n}U^{(j)}_i$.
Then the limsup sets
\begin{equation}\label{eqn15-0}
 \mathcal{U}=\limsup\limits_{i\to\infty}U_i\qquad\text{ and }\qquad
 \mathcal{B}=\limsup\limits_{i\to\infty}B_i
\end{equation}
have the same $\mu$-measure.
\end{lemma}

The key ingredients in the proof of Lemma \ref{measure_unchanged_rectangles} are Lemma~\ref{measure_unchanged} and Fubini's Theorem, which we recall below in the special case of integrating the characteristic function of a measurable set, see
\cite[p. 233]{MR1324786} or \cite[\S2.6.2]{MR0257325}.

\begin{theorem}[Fubini's Theorem] Let $\mu_{1}$ be a $\sigma$-finite measure over $X$ and $\mu_{2}$ be a $\sigma$-finite measure over $Y$. Then $\mu_{1} \times \mu_{2}$ is a regular measure over $X \times Y$. Let $S \subseteq X \times Y$ be a $\mu_{1}\times \mu_{2}$ measurable set and let
\begin{align*}
S^{x}:=\{y : (x,y) \in S \}, \\
S_{y}:=\{x: (x,y) \in S \}.
\end{align*}
Then
\begin{equation*}
(\mu_{1}\times\mu_{2})(S)= \int_Y \mu_{1}(S^{y}) d \mu_{2} = \int_X \mu_{2}(S_{x}) d \mu_{1}.
\end{equation*}
\end{theorem}

We now proceed with the proof of Lemma \ref{measure_unchanged_rectangles}.

\begin{proof}
We initially prove that
\begin{equation*}
\mu \left( \limsup_{i \to \infty} B_{i}^{(1)} \times \prod_{j=2}^{n}B_{i}^{(j)} \right)= \mu \left( \limsup_{i \to \infty} U_{i}^{(1)} \times \prod_{j=2}^{n}B_{i}^{(j)} \right),
\end{equation*}
and note that Lemma \ref{measure_unchanged_rectangles} follows inductively.
For ease of notation let
\begin{equation*}
\hat{\mu}=\prod_{j=2}^{n}\mu_{j}, \quad \hat{B}_{i}=\prod_{j=2}^{n}B_{i}^{(j)}, \quad \hat{X}=\prod_{j=2}^{n}X_{j}.
\end{equation*}
For any $y \in \hat{X}$ let
\begin{equation*}
I_{y}= \{ i: y \in \hat{B}_{i}\},
\end{equation*}
and for any $F \subseteq X$ let $F_y$ denote the fiber of $F$ at $y$, that is
\begin{equation*}
F_y=\{ x: (x,y) \in F\} \subseteq X_{1}.
\end{equation*}
Observe that
\begin{equation}\label{eqn15}
A:=\left( \limsup_{i \to \infty} B_{i}^{(1)} \times \hat{B}_{i} \right)_{y}=\underset{i \in I_y}{\limsup_{i \to \infty}}B^{(1)}_{i}=:D.
\end{equation}
Indeed, if $x \in A$ then it implies there exists an infinite sequence $\{i_{k}\}$ such that
\begin{equation*}
(x,y) \in B_{i_{k}}^{(1)} \times \hat{B}_{i_k} \quad \text{ for all } i_k.
\end{equation*}
Hence $\{i_k\} \subseteq I_y$ and so $x \in D$.

Conversely, if $x \in D$ then $D$ is non-empty and so $I_y$ must be infinite. By the definition of $I_y$ and the fact that $x \in D$ we have that $x \in B^{(1)}_{i}$ for infinitely many $i \in I_y$, and so $x \in A$. \par

Similarly, we have that
\begin{equation}\label{eqn16}
\left( \limsup_{i \to \infty} U_{i}^{(1)} \times \hat{B}_{i} \right)_{y}=\underset{i \in I_y}{\limsup_{i \to \infty}}U^{(1)}_{i}.
\end{equation}

Applying Fubini's Theorem we have that
\begin{align*}
\mu \left( \limsup_{i \to \infty} B_{i}^{(1)} \times \hat{B}_{i} \right) & = \int_{\hat X} \mu_{1} \left( \left( \limsup_{i \to \infty} B_{i}^{(1)} \times \hat{B}_{i} \right)_{y} \right) d \hat{\mu}, \\
& \stackrel{\eqref{eqn15}}{=} \int_{\hat X} \mu_{1} \left( \underset{i \in I_y }{\limsup_{i \to \infty}} B_{i}^{(1)} \right) d \hat{\mu},\\
& \overset{\text{Lemma~\ref{measure_unchanged}}}{=} \int_{\hat X} \mu_{1} \left( \underset{i \in I_y }{\limsup_{i \to \infty}}\; U_{i}^{(1)}  \right)_{y} d \hat{\mu}, \\
& \stackrel{\eqref{eqn16}}{=} \int_{\hat X} \mu_{1} \left( \left( \limsup_{i \to \infty} U_{i}^{(1)} \times \hat{B}_{i} \right)_{y} \right) d \hat{\mu}, \\
&=\mu \left( \limsup_{i \to \infty} U_{i}^{(1)} \times \hat{B}_{i} \right)\,.
\end{align*}
Note that in the above argument we have not made use of the fact $\hat B_i$ are products of balls; we only used the fact that these are measurable sets. Hence, the above argument can be repeated $n-1$ more times, for $\ell=2,\dots,n-1$ each time replacing $B^{(\ell)}_i$ by $U^{(\ell)}_i$ so that at step $\ell$ we get that
\begin{align*}
\mu \left( \limsup_{i \to \infty} \prod_{j=1}^{\ell-1} U_{i}^{(j)} \times \prod_{j=\ell}^n B_{i}^{(j)} \right) & =\mu \left( \limsup_{i \to \infty}\prod_{j=1}^\ell  U_{i}^{(j)} \times \prod_{j=1}^{\ell+1}  B_{i}^{(j)} \right)\,.
\end{align*}
Putting all these equations for $\ell=1,\dots,n$ together we get \eqref{eqn15-0} as claimed.
\end{proof}

Lemma~\ref{borel_cantelli_divergence} only proves positive measure for a $\limsup$ set. In the context of Theorem~\ref{wkpadic} we need a Zero-One law. In \cite{H10} Haynes proved a zero-full result for the simultaneous case. We adapt this method of proof for the weighted case.
\begin{lemma} \label{haynes_0-1}
Let $n \in \N$, $p$ be a prime and $\Psi=(\psi_1,\dots,\psi_n)$ be any $n$-tuple of approximation functions. Then
\begin{equation*}
\mu_{p,n}(\Wp_{n}'(\Psi)) \in \{0,1\}.
\end{equation*}
\end{lemma}

\begin{remark}\rm
Haynes proved this result for the more general setting of $S$-arithmetic approximation. We note that Lemma~\ref{haynes_0-1} can also be proven in the $S$-arithmetic setting, however we limit ourselves to the $p$-adic case to avoid introducing further notation and concepts which are not dealt with in this paper.
\end{remark}

\begin{proof}
Firstly, note that the sets $\Ap_{a_{0}}'(\Psi)$ used to construct our $\limsup$ set have the property that if $p \mid a_{0}$, then $\Ap_{a_{0}}'(\Psi)= \emptyset$ or $\Zp^{n}$, so assume $p \nmid a_{0}$. Define the map $\pi: \Zp \to \Zp$ as follows. For a $p$-adic integer $x \in \Zp$ with $p$-adic expansion
\begin{equation*}
x=\sum_{i=0}^{\infty}a_{i}p^{i}, \quad a_{i} \in \{0, \dots , p-1\},
\end{equation*}
define
\begin{equation*}
\pi(x)= \begin{cases}
\sum_{i=0}^{\infty}a_{i+1}p^{i}, \quad \text{ if } a_{0}=0, \\
1+ \sum_{i=0}^{\infty} a_{i+1}p^{i}, \quad \text{ otherwise}.
\end{cases}
\end{equation*}
Let $\pi_{n}:\Zp^{n} \to \Zp^{n}$ be the transformation $(x_{1}, \dots, x_{n}) \mapsto (\pi(x_{1}), \dots , \pi(x_{n}))$. By using the fact that $p \nmid a_{0}$, and that each $(a_{i},a_{0})=1$, it can be shown that under such mapping
\begin{equation*}
\pi_{n}( \Wp_{n}'(\Psi)) \subseteq \Wp_{n}'(p\Psi),
\end{equation*}
where $p\Psi$ means each component of $\Psi$ has to be multiplied by $p$. This can be repeated inductively to show that $\pi_{n}^{K}( \Wp_{n}'(\Psi)) \subseteq \Wp_{n}'(p^{K}\Psi)$ for any $K \in \N$. Assuming that $\mu_{p,n}(\Wp_{n}'(\Psi))>0$, then by a $p$-adic version of the Lebesgue Density Theorem (see e.g. Lemma~1 in \cite[Part~II, Ch.~1]{S67}) for any $\epsilon>0$ there exists integer vector $\bx_{0} \in \Z^{n}$ and $N \in \N$ such that
\begin{equation*}
\mu_{p,n}\left( \{ \bx \in \Wp_{n}'(\Psi): |\bx-\bx_{0}|_{p} \leq p^{-N} \} \right) \geq (1-\epsilon)p^{-N}.
\end{equation*}
Further, we have that
\begin{equation*}
\pi_{n}^{N}\left( \{ \bx \in \Wp_{n}'(\Psi): |\bx-\bx_{0}|_{p} \leq p^{-N} \} \right) \subseteq \Wp_{n}'(p^{N} \Psi),
\end{equation*}
and so
\begin{align*}
\mu_{p,n} \left( \Wp_{n}'(p^{N}\Psi) \right) & \geq \mu_{p,n} \left( \pi_{n}^{N}\left( \{ \bx \in \Wp_{n}'(\Psi): |\bx-\bx_{0}|_{p} \leq p^{-N} \} \right) \right), \\
&\geq p^{N}(1-\epsilon) p^{-N}, \\
&=(1- \epsilon).
\end{align*}
Since $\epsilon$ is arbitrary we have that $\mu_{p,n}(\bigcup_{N=1}^\infty\Wp_{n}'(p^{N}\Psi))=1$. Now observe that
$$
\Wp_{n}'(\Psi)\subset \Wp_{n}'(p\Psi)\subset \Wp_{n}'(p^{2}\Psi)\subset\dots
$$
and so, by Lemma~\ref{measure_unchanged_rectangles} with $X=\Z_p^n$, $d$ given by the sup norm, and $\mu=\mu_{p,n}$, we have that $\mu_{p,n}(\Wp_{n}'(\Psi))=\mu_{p,n}(\Wp_{n}'(p^N\Psi))$
for every $N\in\N$. Hence,
$$
\mu_{p,n}(\Wp_{n}'(\Psi))=\lim_{N\to\infty}\mu_{p,n}\left(\Wp_{n}'(p^{N}\Psi)\right)=\mu_{p,n}\left(\bigcup_{N=1}^\infty\Wp_{n}'(p^{N}\Psi)\right)=1\,,
$$
thus finishing the proof.
\end{proof}

\section{Proof of Theorems~\ref{wkpadic} \&{} \ref{DS_wkpadic}}

By Lemmas \ref{borel_cantelli_convergence}--\ref{borel_cantelli_divergence} it is clear that we need bounds on the measure of $\Ap'_{a_{0}}(\Psi)$ and $\Ap'_{a_{0}}(\Psi) \cap \Ap'_{b_{0}}(\Psi)$ for $a_{0}, b_{0} \in \N$. As we are considering these measures at fixed values of $a_{0}$ and $b_{0}$ the monotonicity condition of Theorem~\ref{wkpadic} does not appear until we consider the summations over the measures of these sets. For that reason Theorems~\ref{wkpadic} \&{} \ref{DS_wkpadic} are proven in tandem up to such point. \par
Since $(a_{0},a_{i})=1$ observe that we must have $p \nmid a_{0}$. If $p\mid a_{0}$ then the reduced fractions $\frac{a_{i}}{a_{0}}$ used in the composition of $\Ap_{a_{0}}'(\Psi)$ would satisfy $\left| \frac{a_{i}}{a_{0}} \right|_{p}>1$ for any component $1 \leq i \leq n$. And so for sufficiently large $a_0$ we have that
\begin{equation*}
\left\{ \bx \in \Zp^{n}: \left| x_{i}-\frac{a_{i}}{a_{0}} \right|_{p}< \psi_{i}(a_{0}), \, \, 1 \leq i \leq n \right\} = \emptyset,
\end{equation*}
Hence without loss of generality when considering the measure of $\Ap'_{a_{0}}(\Psi)$ and $\Ap'_{a_{0}}(\Psi) \cap \Ap'_{b_{0}}(\Psi)$ we will assume that $p \nmid a_{0},b_{0}$. \par
With regards to the condition that each $\psi_{i}(q)\ll \frac{1}{q}$ note that Lemma \ref{measure_unchanged_rectangles} allows us to reduce this to the condition that each $\psi_{i}(q) < \frac{1}{q}$ for $1 \leq i \leq n$ and the measure results will remain unchanged. Similarly such constants adjustment would not effect the convergence or divergence of the sums of interest.

Note that for any $x \in \Zp$ and $0<r<1$ there exists $t \in \N_{0}$ such that $B(x,r)=B(x,p^{-t})$. For each $1 \leq i \leq n$ define the function $t_{i}:\N \to \N_{0}$ with $t_{i}(a_{0})$ satisfying
\begin{equation*}
p^{-t_{i}(a_{0})} < \psi_{i}(a_{0}) \leq p^{-t_{i}(a_{0})+1}.
\end{equation*}
Then for any $1 \leq i \leq n$ and $a_{0} \in \N$ we have that $\psi_{i}(a_{0}) \asymp p^{-t_{i}(a_{0})}$ and
\begin{equation*}
B\left( x,\psi_{i}(a_{0}) \right)=B\left( x, p^{-t_{i}(a_{0})+1} \right).
\end{equation*}
Hence, without loss of generality we could replace the $n$-tuple of approximation functions $\Psi$ with the function $T$ given by $T(a_0)=\left( p^{-t_{1}(a_{0})+1}, \dots , p^{-t_{n}(a_{0})+1} \right)$. Thus, we have that $\mu_{p,n}(\Wp_{n}(\Psi))=\mu_{p,n}(\Wp_{n}(T))$.

For $a_{0},b_{0} \in \N$ and $\varphi$ Euler's totient function we prove the following claims
 \begin{enumerate}[(a)]
 \item $\mu_{p,n}\left( \Ap'_{a_{0}}(\Psi)\right) \ll \varphi(a_{0})^{n}\prod_{i=1}^{n}\psi_{i}(a_{0})$,
 \item $\mu_{p,n}\left( \Ap'_{a_{0}}(\Psi)\right) \gg \varphi(a_{0})^{n}\prod_{i=1}^{n}\psi_{i}(a_{0})$,
 \item $\mu_{p,n}\left( \Ap_{a_{0}}(\Psi) \cap \Ap_{b_{0}}(\Psi) \right) \ll a_{0}^{n}b_{0}^{n} \prod_{i=1}^{n}\psi_{i}(a_{0})\psi_{i}(b_{0})$.
 \end{enumerate}

Beginning with (a) observe that
\begin{equation}\label{disjoint_measure0}
\mu_{p,n}(\Ap'_{a_{0}}(\Psi)) =\mu_{p,n} \left( \underset{\gcd(a_{i},a_{0})=1, \, 1 \leq i \leq n}{\bigcup_{|a_{i}| \leq a_{0}}} \prod_{i=1}^{n}B\left( \frac{a_{i}}{a_{0}}, \psi_{i}(a_{0}) \right) \right).
\end{equation}
If each rectangle in the above composition is disjoint then
\begin{align}
\mu_{p,n}(\Ap'_{a_{0}}(\Psi))& =\underset{\gcd(a_{i},a_{0})=1, \, 1 \leq i \leq n}{\sum_{|a_{i}| \leq a_{0}}}\mu_{p,n} \left(  \prod_{i=1}^{n}B\left( \frac{a_{i}}{a_{0}}, \psi_{i}(a_{0}) \right) \right)
\asymp \varphi(a_{0})^{n}\prod_{i=1}^{n}\psi_{i}(a_{0}), \label{disjoint_measure}
\end{align}
since $\mu_{p,n}$ is the product measure of $n$ copies of $\mu_{p}$, and so the measure of the product of the balls in the above expression equals the product of their measures. This provides us with an upper bound on $\mu_{p,n}(\Ap'_{a_{0}}(\Psi))$, since any non-empty intersections in the union within \eqref{disjoint_measure0} would only make the measure of the union smaller than their sum given by \eqref{disjoint_measure}.

To prove (b) we simply need to show that the union within \eqref{disjoint_measure0} contains no non-empty intersections. Suppose this is not the case, say
\begin{equation*}
\left( \prod_{i=1}^{n}B\left(\frac{b_{i}}{a_{0}}, \psi_{i}(a_{0})\right)\right) \bigcap \left( \prod_{i=1}^{n}B\left(\frac{c_{i}}{a_{0}}, \psi_{i}(a_{0})\right)\right) \neq \emptyset,
\end{equation*}
for some points $b=(b_{1}, \dots , b_{n}),c=(c_{1}, \dots , c_{n}) \in \Z^{n}$ with $|b_{i}|,|c_{i}| \leq a_{0}$ and $b \neq c$. Then we have that
\begin{equation*}
\left|b_{i}-c_{i}\right|_{p} \leq \psi_{i}(a_{0}), \quad 1 \leq i \leq n,
\end{equation*}
since $|a_{0}|_{p}=1$. Such inequalities would hold if and only if $\psi_{i}(a_{0})\geq \frac{1}{a_{0}}$ for all $1 \leq i \leq n$ such that $b_{j} \neq c_{j}$. However, we have that $\psi_{i}(q)< \frac{1}{q}$ for all $1 \leq i \leq n$ and $q\in\N$ and thus, by \eqref{disjoint_measure}, we have the required lower bound on $\mu_{p,n}(\Ap'_{a_{0}}(\Psi))$.

To prove (c) define the set
 \begin{equation*}
Q:= \{(a,b) \in \Z^{2}: |a| \leq a_{0}, \, |b| \leq b_{0}, \, \gcd(a,a_{0})=\gcd(b,b_{0})=1 \}.
\end{equation*}
Observe that
 \begin{equation} \label{Ap_intersect}
 \mu_{p,n}\left(\Ap'_{a_{0}}(\Psi) \cap \Ap'_{b_{0}}(\Psi)\right) \ll \prod_{i=1}^{n}\# \left\{(a_{i},b_{i}) \in Q:  \left| \frac{a_{i}}{a_{0}}-\frac{b_{i}}{b_{0}} \right|_{p} <  \Delta_i \right\} \delta_i.
\end{equation}
where
$$
\Delta_i=\max\{\psi_{i}(a_{0}),\psi_{i}(b_{0})\}\qquad\text{and}\qquad \delta_i=\min\{\psi_{i}(a_{0}),\psi_{i}(b_{0})\}\,.
$$
Fix any $i$ and without loss of generality suppose that $\Delta_i=\psi_{i}(a_{0}) \geq \psi_{i}(b_{0})=\delta_i$.
Note that since $p \nmid a_{0},b_{0}$ then the inequality in the above equation is equivalent to $(a_{i},b_{i}) \in Q$ satisfying
\begin{equation} \label{xxx}
|a_{i}b_{0}-b_{i}a_{0}|_{p} < \psi_{i}(a_{0}).
\end{equation}
To count solutions satisfying \eqref{xxx} we observe that such solutions also solve the congruence
\begin{equation} \label{modular}
b_{i}a_{0}-a_{i}b_{0} \equiv 0 \mod p^{t_{i}(a_{0})}.
\end{equation}
Let $d=\gcd(a_{0},b_{0})$, and let $a_{0}'=\frac{a_{0}}{d}$ and $b_{0}'=\frac{b_{0}}{d}$. Suppose that
\begin{equation*}
b_{i}a_{0}'-a_{i}b_{0}'=k,
\end{equation*}
for some integer $k$, with $|k| \leq 2\frac{a_{0}b_{0}}{d}$. The bounds on $k$ follow on the observation that
\begin{equation*}
|b_{i}a_{0}-a_{i}b_{0}| \leq 2a_{0}b_{0},
\end{equation*}
for all $(a_{i},b_{i}) \in Q$. Considering the congruence
\begin{equation*}
a_{i}b_{0}' \equiv b_{i}a_{0}'-k \mod a_{0}',
\end{equation*}
note that per $k$ there is at most one solution $a_{i}$ modulo $a_{0}'$, and so at most $\frac{2a_{0}}{a_{0}'}=2d$ possible $a_{i}$ with $|a_{i}| \leq a_{0}$. Clearly, each $b_{i}$ is uniquely determined by each $a_{i}$ and $k$, so per fixed $k$ there are at most $2d$ possible pairs $(a_{i},b_{i}) \in Q$. To solve \eqref{modular} we must have that
\begin{equation} \label{kzero}
k \equiv 0 \mod p^{t_{i}(a_{0})},
\end{equation}
of which there are at most
\begin{equation*}
\frac{4a_{0}b_{0}}{d\,p^{t_{i}(a_{0})}}+1
\end{equation*}
possible $k$ satisfying $|k| \leq 2a_{0}b_{0}/d$. Note that one such possible value of $k$ satisfying \eqref{kzero} is $k=0$. But this is impossible, since it implies that
\begin{equation*}
a_{0}'b_{i}=a_{i}b_{0}'.
\end{equation*}
Indeed, assuming $a_{0}>b_{0}$, we get that $a_{0} \neq 1$ and $\gcd(a_{0}', a_{i})=\gcd(a_{0}', b_{0}')=1$, and so we must have that $b_{i}a_{0}'-a_{i}b_{0}' \neq 0$. If $b_{0}>a_{0}$ then the argument is similar. Hence there are at most
\begin{equation*}
\frac{4a_{0}b_{0}}{d\,p^{t_{i}(a_{0})}}
\end{equation*}
values of $k$ that have corresponding solutions in $Q$, and so there are at most
\begin{equation*}
2d\frac{4a_{0}b_{0}}{d\,p^{t_{i}(a_{0})}}\ll a_{0}b_{0}\psi_{i}(a_{0})
\end{equation*}
pairs $(a_{i},b_{i}) \in Q$ that solve \eqref{xxx}. Combining this upper bound with \eqref{Ap_intersect} we have that
\begin{equation*}
\mu_{p,n}( \Ap'_{a_{0}}(\Psi) \cap \Ap'_{b_{0}}(\Psi) ) \ll a_{0}^{n}b_{0}^{n}\prod_{i=1}^{n} \psi_{i}(a_{0})\psi_{i}(b_{0}).
\end{equation*}
By (c), we have that
\begin{equation} \label{intersect_sum}
\sum_{a_{0},b_{0}=1}^{N} \mu_{p,n}(\Ap'_{a_{0}}(\Psi) \cap \Ap'_{b_{0}}(\Psi)) \ll \sum_{a_{0},b_{0}=1}^{N}a_{0}^{n}b_{0}^{n} \prod_{i=1}^{n} \psi_{i}(a_{0})\psi_{i}(b_{0}) \ll \left( \sum_{a_{0}=1}^{N} a_{0}^{n}\prod_{i=1}^{n}\psi_{i}(a_{0}) \right)^{2}.
\end{equation}
Now assuming the monotonicity of $\prod_{i=1}^{n}\psi_{i}(q)$, by (a), (b), we have that
\begin{equation}\label{eq20}
  \sum_{a_0=1}^N \mu_{p,n}(\Ap_{a_{0}}'(\Psi)) \asymp \sum_{a_0=1}^N \varphi(a_{0})^{n}\prod_{i=1}^{n}\psi_{i}(a_{0})
  \asymp \sum_{a_0=1}^N a_{0}^{n}\prod_{i=1}^{n}\psi_{i}(a_{0})\,.
\end{equation}
Hence \eqref{eq20} completes the convergence case of Theorem~\ref{wkpadic} via Lemma \ref{borel_cantelli_convergence}. In turn, \eqref{intersect_sum} and \eqref{eq20} together with Lemma \ref{borel_cantelli_divergence} proves that $\mu_{p,n}(\Wp_{n}'(\Psi))>0$ and finally applying Lemma \ref{haynes_0-1} completes the proof of Theorem~\ref{wkpadic}.

Regarding Theorem~\ref{DS_wkpadic}, Claim (a), completes the convergence case via Lemma \ref{borel_cantelli_convergence}. In the divergence case we note that Claim (b), \eqref{intersect_sum} and condition \eqref{DS_condition} imply that
\begin{equation*}
\limsup_{N \to \infty} \left( \frac{ \sum_{a_0=1}^N \varphi(a_{0})^{n}\prod_{i=1}^{n}\psi_{i}(a_{0})}{\sum_{a_{0}=1}^{N} a_{0}^{n}\prod_{i=1}^{n}\psi_{i}(a_{0})} \right)^{2}>0.
\end{equation*}
Hence, Lemma~\ref{borel_cantelli_divergence} is applicable and we get that $\mu_{p,n}(\Wp_{n}'(\Psi))>0$.
Applying Lemma~\ref{haynes_0-1} completes the proof of Theorem~\ref{DS_wkpadic}.

\section{Proof of Theorem \ref{weighted_jb}}

As with many Hausdorff dimension results we prove the upper bound and lower bound independently. As we are working with $\limsup$ sets of hyperrectangles defined by \eqref{Ap} we will naturally appeal to Theorem \ref{MTPRR} to get the lower bound. We start with the upper bound which takes advantage of a standard cover of $\Wp_{n}(\btau)$.

\medskip

\noindent\textit{Upper bound in Theorem \ref{weighted_jb}}.
Recall that
$
\Wp_{n}(\Psi)=\limsup_{a_{0} \to \infty} \Ap_{a_{0}}(\Psi)\,,
$
where $\Ap_{a_{0}}(\Psi)$ is given by \eqref{Ap}, that is
\begin{equation*}
\Ap_{a_{0}}(\Psi)= \bigcup_{\substack{(a_1,\dots,a_n) \in\Z^n\\ |a_{i}| \leq a_{0}\; (1\le i\le n)}} R_{a_{0},a_1,\dots,a_n}(\Psi)
\end{equation*}
and
$$
R_{a_{0},a_1,\dots,a_n}(\Psi)= \left\{\bx=(x_1,\dots,x_n) \in \Zp^{n}: \left|x_{i}-\frac{a_{i}}{a_{0}} \right|_{p} < \psi_{i}(a_{0})\;\;\text{for }1\leq i \leq n \right\}.
$$
Throughout this proof $\Psi=(q^{-\tau_1},\dots,q^{-\tau_n})$. Then for every $i\in\{1,\dots,n\}$ we can trivially cover $R_{a_{0},a_1,\dots,a_n}(\btau):=R_{a_{0},a_1,\dots,a_n}(\Psi)$ by a finite collection $\mathfrak{B}(a_{0})$ of balls of radius $a_{0}^{-\tau_{i}}$ such that
\begin{equation*}
\# \mathfrak{B}(a_{0}) \ll \prod_{j=1}^{n} \max \left\{ 1, \frac{a_{0}^{-\tau_{j}}}{a_{0}^{-\tau_{i}}}\right\}= a_{0}^{\sum_{\tau_{j}<\tau_{i}}(\tau_{i}-\tau_{j})},
\end{equation*}
where the power of $a_{0}$ on the R.H.S of the above inequality can be obtained by removing the cases where $\frac{a_{0}^{-\tau_{j}}}{a_{0}^{-\tau_{i}}}<1$. Let $s_{0}=\frac{n+1+\sum_{\tau_{j} < \tau_{i}}(\tau_{i}-\tau_{j})+ \delta}{\tau_{i}}$ for some $\delta>0$. Then for any $N>0$
\begin{align*}
\h^{s_{0}}(\Wp_{n}(\btau)) & \ll \sum_{a_{0} \geq N} \underset{1 \leq i \leq n}{\underset{|a_{i}| \leq a_{0}}{\sum}} \# \mathfrak{B}(a_{0})a_{0}^{-s_{0}\tau_{i}}, \\
& \ll \sum_{a_{0} \geq N} a_{0}^{n+ \sum_{\tau_{j}<\tau_{i}}(\tau_{i}-\tau_{j})-s_{0}\tau_{i}}, \\
&=\sum_{a_{0} \geq N} a_{0}^{-1-\delta}\to0\quad\text{as }N\to \infty.
\end{align*}
This implies that $\dim \Wp_{n}(\btau))\le s_0$. The above argument follows for any choice of $\tau_{i}$, hence we may choose the minimum over the set of all $\tau_{i}$ and so the upper bound for the dimension in Theorem~\ref{weighted_jb} follows on letting $\delta\to0$.

\medskip

\noindent\textit{Lower bound in Theorem \ref{weighted_jb}}. In order to apply Theorem \ref{MTPRR} we need to construct a set of resonant points that we can show are a locally ubiquitous system of rectangles. Let
\begin{equation*}
\hat R_{a_{0},i}= \left\{ \frac{a_{i}}{a_{0}} \in \Q : |a_{i}| \leq a_{0} \right\},
\end{equation*}
for each $1 \leq i \leq n$, and let $\hat R_{a_{0}}= \prod_{i=1}^{n} R_{a_{0},i}$. In line with the notation prior to Theorem \ref{MTPRR} let $J=\N$, and $\beta:J \to \R_{+}$ be $\beta(a_{0})=a_{0}$. Choose $\rho:\R_{+} \to \R_{+}$ to be $\rho(a_{0})=a_{0}^{-1}$, and choose the two sequences $l_{k}=M^{k}$, and $u_{k}=M^{k+1}$, for some fixed integer $M \ge2$ to be chosen later, so that
\begin{equation*}
J_{k}= \{ a_{0} \in \N: M^{k} \leq a_{0} \leq M^{k+1} \}.
\end{equation*}
In order to show such set of resonant points is a local ubiquitous system of rectangles we prove the following proposition.

\begin{proposition} \label{ubiquity_rectangle}
Let $\hat R_{a_{0}}$, $J$, $\beta$, and $\rho$ be defined as above. Let $\bal=(\alpha_{1}, \dots, \alpha_{n}) \in \R^{n}_{+}$ with each $\alpha_{i}>1$ be a vector satisfying
\begin{equation}\label{eqn31}
\sum_{i=1}^{n} \alpha_{i}=n+1.
\end{equation}
There are constants $M>1$ and $c_1>0$ such that for any ball $B \subset \Zp^{n}$
\begin{equation*}
\mu_{p,n} \left( B \cap \bigcup_{M^{k} \leq a_{0} \leq M^{k+1}} \Delta\left(R_{a_{0}}, \left( \frac{c_{1}}{M^{k+1}} \right)^{\bal} \right) \right) \geq \tfrac12\, \mu_{p,n}(B)
\end{equation*}
for all sufficiently large $k \in \N$.
\end{proposition}

\begin{proof}
Fix some ball $B=B(\by,r)$ for some $\by \in \Zp^{n}$ and $r \in \{p^{i}: i \in \N\cup\{0\} \}$. We will assume that $k$ is sufficiently large so that $M^{k}r\geq 1$. In view of \eqref{eqn31} and the fact that $\alpha_i>1$ for all $i$, by Lemma~\ref{mink}, we have that for any $\bx=(x_{1}, \dots, x_{n}) \in B$ there exists $(a_{0},\dots,a_{n}) \in \Z^{n+1}$, satisfying
\begin{equation}\label{eqn+2}
|a_{i}| \;\le\; M^{k}\qquad(1 \leq i \leq n),\qquad 0<a_0< M^{k+1}
\end{equation}
and
\begin{equation} \label{eq24}
|a_{0}x_{i}-a_{i}|_{p}<p\left(M^{k+\frac{1}{n+1}}\right)^{-\alpha_{i}}\qquad(1\leq i \leq n).
\end{equation}
Since $\alpha_{i}>1$ for each $1 \leq i \leq n$, \eqref{eq24} combined with $0<a_{0} \leq M^{k+1}$ implies that $|a_{i}|_{p} \leq |a_{0}|_{p}$ for each $1 \leq i \leq n$, provided that $k$ is sufficiently large. Let $\lambda$ be the integer such that $|a_{0}|_{p}=p^{-\lambda}$. Write $a_{0}'=a_{0}p^{-\lambda}$ and $a_{i}'=a_{i}p^{-\lambda}$. Since $|a_i|_p\le |a_0|_p$, we have that $a_{0}',a_{i}' \in \Z$. Also, by definition,
\begin{equation} \label{reduced_points}
 0<a'_{0}\leq p^{-\lambda}M^{k+1}, \quad |a'_{i}|\le p^{-\lambda}M^{k}
\end{equation}
for each $1 \leq i \leq n$ and, by \eqref{eq24} and the fact that $\gcd(a'_0,p)=1$, that
\begin{align}
\left|x_{i}-\frac{a_{i}'}{a_{0}'}\right|_{p}&=\left|a'_0x_{i}-{a_{i}'}\right|_{p}=p^{\lambda}|a_{0}x_{i}-a_{i}|_{p}
<p^{\lambda+1}\left(M^{k+\frac{1}{n+1}}\right)^{-\alpha_{i}} \label{eq25}
\end{align}
for $1 \leq i \leq n$. We want to remove the $a_{0}'$ values that are divisible by too high a power of $p$, that is $|a_{0}'|_{p}  < p^{-\lambda_{0}}$ for some fixed $\lambda_{0} \in \N$ to be chosen later. We consider the integer vectors $(a_{0}', \dots , a_{n}')$ satisfying \eqref{reduced_points} such that
\begin{equation*}
\left( \frac{a_{1}'}{a_{0}'}, \dots , \frac{a_{n}'}{a_{0}'} \right) \in B(\by,r)\,,
\end{equation*}
where $r$ is a poser of $p$.
Considering the congruence equations for $a_{0}'$ fixed we have that there are
\begin{equation*}
\left(2p^{-\lambda}M^{k}r+1 \right)^{n}< \left(3 p^{-\lambda} M^{k}r \right)^{n}
\end{equation*}
such points. Hence
\begin{align*}
\mu_{p,n} \left( B \cap \underset{\lambda \geq \lambda_{0}}{\bigcup} \underset{ 0< a_{0}' \leq \frac{M^{k+1}}{p^{\lambda}}}{\bigcup_{|a'_{i}|\leq \frac{M^{k}}{p^{\lambda}}}} \bigcup_{\left( \frac{a_{1}'}{a_{0}'}, \dots , \frac{a_{n}'}{a_{0}'} \right) \in B} \prod_{i=1}^{n} B\left( \frac{a_{i}'}{a_{0}'}, p^{\lambda+1}\left(M^{k+\frac{1}{n+1}}\right)^{-\alpha_{i}} \right) \right)&\\
& \hspace{-6cm}\leq \sum_{\lambda \geq \lambda_{0}} \frac{M^{k+1}}{p^{\lambda}} \left(3 \frac{M^{k}}{p^{\lambda}}r \right)^{n} p^{n \lambda + n} M^{-k(n+1)-1}, \\
&  \hspace{-4cm} = \sum_{\lambda \geq \lambda_{0}} \mu_{p,n}(B)3^{n}p^{n-\lambda}, \\
&  \hspace{-2cm} \leq 3^{n}\frac{p^{n+1-\lambda_{0}}}{p-1}\mu_{p,n}(B).
\end{align*}
Taking $\lambda_{0}$ sufficiently large, e.g. $p^{\lambda_{0}} > 4\frac{3^{n}p^{n+1}}{p-1}$, then we have that
\begin{equation}\label{eqn+0}
\mu_{p,n} \left( B \cap \underset{\lambda \geq \lambda_{0}}{\bigcup} \underset{ 0< a_{0}' \leq \frac{M^{k+1}}{p^{\lambda}}}{\bigcup_{|a_{i}|\leq \frac{M^{k}}{p^{\lambda}}}} \bigcup_{\left( \frac{a_{1}'}{a_{0}'}, \dots , \frac{a_{n}'}{a_{0}'} \right) \in B} \prod_{i=1}^{n} B\left( \frac{a_{i}'}{a_{0}'}, p^{\lambda_{0}+1}\left(M^{k+\frac{1}{n+1}}\right)^{-\alpha_{i}} \right) \right) \leq \tfrac{1}{4} \mu_{p,n}(B).
\end{equation}
Using similar calculations to those of above we have that
\begin{equation}\label{eqn+1}
\mu_{p,n}\left( B \cap\hspace*{-3ex} \underset{0 < a_{0} \leq M^{k}: |a_{0}|_{p} \geq p^{-\lambda_{0}}}{\bigcup_{|a_{i}| \leq M^{k}}} \hspace*{-3ex}\prod_{i=1}^{n} B\left( \frac{a_{i}}{a_{0}} , p^{\lambda_{0}+1}\left(M^{k+\frac{1}{n+1}}\right)^{-\alpha_{i}} \right) \right) \leq \frac{3^{n}p^{n\lambda_{0}+n}}{M} \mu_{p,n}(B)\le\tfrac14\mu_{p,n}(B)
\end{equation}
provided that $M>4 \, 3^{n} p^{n\lambda_{0}+n}$. Combining \eqref{eqn+0} and \eqref{eqn+1} with the fact that \eqref{eqn+2} and \eqref{eq24} can be solved in integers $(a_0,\dots,a_n)$ for all $\bx$, we get that
\begin{equation*}
\mu_{p,n} \left( B \cap \underset{M^{k} < a_{0} \leq M^{k+1}: |a_{0}|_{p} \geq p^{-\lambda_{0}}}{\bigcup_{|a_{i}| \leq M^{k} }}\prod_{i=1}^{n}B\left( \frac{a_{i}}{a_{0}}, p^{\lambda_{0}+1}\left(M^{k+\frac{1}{n+1}}\right)^{-\alpha_{i}} \right) \right) \geq \tfrac12\, \mu_{p,n}(B).
\end{equation*}
Taking the constant
\begin{equation*}
c_{1}=\max_{1 \leq i \leq n} \, p^{\frac{\lambda_{0}+1}{\alpha_{i}}}M^{1-\frac{1}{n+1}}
\end{equation*}
completes the proof.
\end{proof}

Given Proposition \ref{ubiquity_rectangle}, we have that $(R_{a_{0}}, \beta)$ is a local ubiquitous system with respect to $(\rho,\bal)$, provided $\sum_{i=1}^{n}\al_{i}=n+1$. Using the setup provided for Theorem \ref{MTPRR} let $\btau=(\tau_{1}, \dots, \tau_{n})=(\al_{1}+t_{1}, \dots, \al_{n}+t_{n}) \in \R^{n}_{+}$, then $\Wp_{n}(\btau)=W(\bt)$. Without loss of generality let $\tau_{1} \geq \dots \geq \tau_{n}$. Define $\al_{i}$ recursively as
\begin{equation*}
\al_{i}=\min \left\{ \tau_{i}, \frac{n+1-\sum_{j=n-i+1}^{n}\al_{j}}{n-i} \right\}.
\end{equation*}
Since $\sum_{i=1}^{n}\tau_{i}>n+1$ and $\sum_{i=1}^{n}\al_{i}=n+1$ such recursive formula is possible and we have that $\al_{i} \leq \tau_{i}$ for each $1 \leq i \leq n$, so $\bt$ is well defined. Since $\tau_{1} \geq \dots \geq \tau_{n}$ we have that $\al_{1} \geq \dots \geq \al_{n}$, and furthermore there exists $k \in \{1, \dots, n\}$ such that for all $1 \leq i \leq n-k$
\begin{equation*}
\al_{i}=\frac{n+1-\sum_{j=n-k+1}^{n}\al_{j}}{n-k}.
\end{equation*}
Such observation follows by noting that at least
\begin{equation*}
\al_{1}=n+1-\sum_{j=n-1}^{n}\al_{j}
\end{equation*}
by the fact that $\sum_{i=1}^{n}\al_{i}=n+1$. Note that for each metric space $X_{i}=\Zp$ the Haar measure $\mu_{p}$ is a $1$-Ahlfors probability measure. With reference to Theorem \ref{MTPRR}, consider the following three cases:
\begin{enumerate}[i)]
\item $A_{i} \in \{ \al_{1}, \dots \al_{n-k} \}$: For these values of $A_{i}$ we have that
\begin{equation*}
\begin{array}{ccc}
K_{1}=\{1, \dots , n-k \}, & K_{2}=\{ n-k+1, \dots , n\}, & K_{3}=\emptyset.
\end{array}
\end{equation*}
Applying Theorem \ref{MTPRR} we get that
\begin{align*}
\dim \Wp_{n}(\btau) & \geq \min_{A_{i}} \left\{ \frac{(n-k)\al_{i}+(n-(n-k+1)+1)\al_{i}- \sum_{j=n-k}^{n}t_{j}}{\al_{i}} \right\}, \\
&=\min_{A_{i}} \left\{ n-\frac{\sum_{j=n-k+1}^{n}t_{j}}{\al_{i}} \right\}.
\end{align*}
Since $t_{i}=0$ for $n-k < i \leq n$ we have that $\dim \Wp_{n}(\btau) \geq n$.
\item $A_{i} \in \{ \al_{n-k+1}, \dots , \al_{n} \}$: For such values of $A_{i}$ observe that
\begin{equation*}
\begin{array}{ccc}
K_{1}=\{1, \dots , i \}, & K_{2}=\{i+1, \dots , n\}, & K_{3}=\emptyset.
\end{array}
\end{equation*}
Applying Theorem \ref{MTPRR} we have, in this case,
\begin{equation*}
\dim \Wp_{n}(\btau) \geq \min_{A_{i}}\left\{ \frac{i\al_{i} + (n-i)\al_{i} - \sum_{j=i+1}^{n} t_{j}}{\al_{i}}  \right\}.
\end{equation*}
Similarly to the previous case, since $t_{j}=0$ for $n-k+1 \leq i \leq n$ the r.h.s of the above equation is $n$, the maximal dimension of $\Wp_{n}(\btau)$.
\item $A_{i} \in \{\tau_{1}, \dots , \tau_{n}\}$: Since $\tau_{i}=\al_{i}$ for $n-k+1 \leq i \leq n$, ii) covers such result. So we only need to consider the set of $A_{i} \in \{ \tau_{1}, \dots \tau_{n-k} \}$. If $A_{i}$ is contained in such set, then
\begin{equation*}
\begin{array}{ccc}
K_{1}=\emptyset, & K_{2}=\{ i, \dots , n\}, & K_{3}=\{ 1, \dots , i-1\}.
\end{array}
\end{equation*}
Thus, by Theorem \ref{MTPRR}, we have that
\begin{align*}
\dim \Wp_{n}(\btau) & \geq \min_{A_{i}} \left\{ \frac{(n-i+1)\tau_{i}+\sum_{j=1}^{i-1}a_{j}-\sum_{j=i}^{n}t_{j}}{\tau_{i}} \right\}, \\
&= \min_{A_{i}} \left\{ \frac{(n-i+1)\tau_{i}+(i-1)\left( \frac{n+1-\sum_{j=n-k+1}^{n}a_{j}}{n-k} \right)-\sum_{j=i}^{n-k}(\tau_{j}-a_{j})-\sum_{j=n-k+1}^{n}t_{j}}{\tau_{i}} \right\}, \\
& = \min_{A_{i}}\left\{ \frac{(n-i+1)\tau_{i}+(n-k)\left( \frac{n+1-\sum_{j=n-k+1}^{n}a_{j}}{n-k} \right)-\sum_{j=i}^{n-k}\tau_{j}-\sum_{j=n-k+1}^{n}t_{j}}{\tau_{i}} \right\}, \\
&=\min_{A_{i}} \left\{ \frac{n+1+\sum_{j=i}^{n}(\tau_{i}-\tau_{j})}{\tau_{i}} \right\},
\end{align*}
since $a_{j}+t_{j}=\tau_{j}$.
\end{enumerate}
These are all possible choices of $A_{i}$. The proof of Theorem \ref{weighted_jb} is thus complete.

\section{Dirichlet-style Theorem on $p$-adic manifolds}

This section provides a full measure statement needed to deploy a Mass Transference Principle for the proofs of Theorems~\ref{lowerbnd}--\ref{lowerbnd_manifold}.

\begin{theorem} \label{diri}
Let $\bff=(f_{1}, \dots , f_{m}):\U \to \Zp^{m}$ be a map defined on an open subset $\U \subseteq \Zp^{d}$, $\bx\in\U$ and suppose that $\bff$ is \DQE{} at $\bx$ and let $\lambda$ be given by
\begin{align}
\max\left\{1,\,\underset{1 \leq j \leq m}{\max_{1 \leq i \leq d}} \left| \frac{\partial f_{j}}{\partial x_{i}}(\bx) \right|_{p}\right\}=p^{\lambda}\,. \label{first_order_constant}
\end{align}
Let $\btau=(\tau_{1}, \dots, \tau_{m}) \in \R^{m}_{+}$, $\bv=(v_{1}, \dots , v_{d}) \in \R^{d}_{+}$ and
\begin{equation*}
\begin{array}{cc}
\sum_{i=1}^{m}\tau_{i} < m+1, & \tau_{i} >1, \,\, (1 \leq i \leq m), \\[1ex]
\sum_{i=1}^{d}v_{i}=n+1-\sum_{i=1}^{m}\tau_{i}, & v_{i}>1, \, \, (1 \leq i \leq d)\,.
\end{array}
\end{equation*}
Then there exist $H_{0} \in \N$ such that for all $H > H_{0}$ and some $k\in\Z$ the following system
\begin{equation}
\begin{cases}
\left|x_{i} -\frac{a_{i}}{a_{0}}\right|_{p} < p^{(n+m\lambda)/d}p^kH^{-v_{i}} \qquad\qquad\qquad\quad (1 \leq i \leq d), \\[2ex]
\left|f_{j}\left( \frac{a_{1}}{a_{0}}, \dots , \frac{a_{d}}{a_{0}} \right) - \frac{a_{d+j}}{a_{0}} \right|_{p} < (p^{-k}H)^{-\tau_{j}} \, \qquad\quad (1 \leq j \leq m), \label{tau_{2}} \\[2ex]
\underset{0 \leq i \leq n}{\max}  |a_{i}| \leq p^{-k} H
\end{cases}
\end{equation}
has a solution $(a_{0}, \dots, a_{n}) \in \Z^{n+1}$ satisfying
\begin{equation}\label{a}
\text{$(a_{0},p)=1$,\qquad $\gcd(a_{0}, \dots, a_{n})=1$\qquad and\qquad $\left( \tfrac{a_{1}}{a_{0}}, \dots , \tfrac{a_{d}}{a_{0}} \right)\in\U$.}
\end{equation}
\end{theorem}

\begin{proof}
By Lemma \ref{mink} with $\bsigma=((n+m\lambda)/d,\dots,(n+m\lambda)/d,-\lambda,\dots,-\lambda)$, $H_0=\dots=H_n=H$ and $T=H+1$, for any integer $H \ge H_{\bsigma}^{1/(n+1)}$ the following system
\begin{align}
\begin{cases}
|b_{0}x_{i}-b_{i}|_{p} &< p^{(n+m\lambda)/d}H^{-v_{i}} \quad (1 \leq i \leq d), \\[1.5ex]
\displaystyle\left| b_{0}p^{\lambda}f_{j}(\bx)-\sum_{i=1}^{d}p^{\lambda}\dfrac{\partial f_{j}}{\partial x_{i}}(\bx) \left( b_0x_{i}-b_{i} \right)-p^{\lambda}b_{d+j} \right|_{p} &< p^{-\lambda}H^{-\tau_{j}} \hspace*{8ex} (1 \leq j \leq m), \label{ineq2}\\[1.5ex]
\underset{0\leq i\leq n}{\max} |b_{i}|  & \leq H
\end{cases}
\end{align}
has a non-zero integer solution $(b_{0},b_{1},\dots, b_{n}) \in \Z^{n+1}$. Without loss of generality we can assume that $d=\gcd(b_{0},b_{1},\dots, b_{n})$ is a power of $p$ as otherwise we can divide \eqref{ineq2} through by any other prime powers in the factorisation of $d$ without affecting \eqref{ineq2}. Let $C>0$ and $0<\varepsilon<1$ be the constants that satisfy Definition~\ref{DQE} for all $f_j$ simultaneously. In particular, we have that $B(\bx,\varepsilon) \subseteq \U$. Let $v_{\min}:=\min_{1 \leq i \leq d}v_{i}$ and $\tau_{\max}:=\max_{1 \leq j \leq m} \tau_{j}$.
Let $H_{0}$ be defined as follows
\begin{equation*}
H_{0}:=\max \left\{ \begin{array}{cc}
 C^{\frac{2}{2v_{\min} - \tau_{\max}}}, & (\alpha1) \\
 C^{\frac{1}{v_{\min} - 1}}, & (\alpha2) \\
  (\varepsilon^{-1}p^{(n+m\lambda)/d})^{\frac{1}{v_{\min}}-1}, & (\beta) \\
   p^{\frac{n+n\lambda}{d(v_{\min} -1)}}, & (\gamma)\\
   H_{\bsigma}^{1/(n+1)} & (\delta)
   \end{array} \right\}\,.
\end{equation*}
Note that $H_{0}$ is a well defined positive real number since $v_{\min}-1>0$ and
$2v_{\min}-\tau_{\max}>0$. The latter follows from the facts that each $\tau_{j}>1$ and $\sum_{j=1}^{m}\tau_{j}<m+1$ and so $\tau_{j}<2$, and the condition that each $v_{i}>1$.
Note that $(\gamma)$ implies that $p^{(n+m\lambda)/d}H^{-v_{i}}<H^{-1}$ whenever $H> H_0$. We will use this observation a few times in this proof.

We will now prove two statements concerning the integer solution $(b_{0},b_{1},\dots, b_{n})$ to \eqref{ineq2}.
First we verify that $b_0\neq0$. Suppose the contrary, that is $b_{0}=0$. Then by the first inequality of \eqref{ineq2} we have that $|b_{i}|_{p}<p^{(n+m\lambda)/d}H^{-v_{i}}<H^{-1}$. As $|b_{i}| \leq H$ and $H>H_{0}$, we have that $b_{i}=0$ for $1 \leq i \leq d$. Considering the second set of inequalities of \eqref{ineq2}, for each $1 \leq j \leq m$ we have that $|b_{d+j}|_{p}<H^{-\tau_{j}}$ which also forces us to conclude that $b_{d+j}=0$,  since $\tau_{j}>1$ for each $1 \leq j \leq m$. Thus $(b_{0}, b_{1}, \dots, b_{n})=\boldsymbol{0}$, a contradiction. So we must have that $b_{0}\neq 0$.

Now we show that $\frac{b_{i}}{b_{0}}$ is a $p$-adic integer for all $1 \leq i \leq d$. Since $b_{0} \neq 0$, we may rewrite the first inequality of \eqref{ineq2} to get
 \begin{equation*}
 |b_{0}|_{p}\left| x_{i}-\frac{b_{i}}{b_{0}} \right|_{p} <p^{(n+m\lambda)/d}H^{-v_{i}}, \quad 1 \leq i \leq d.
 \end{equation*}
 Suppose that $\left|\frac{b_{i}}{b_{0}}\right|_{p}>1$ for some $1 \leq i \leq d$, then $\left|\frac{b_{i}}{b_{0}}\right|_{p}>|x_{i}|_{p}$ since $\bx \in \U \subseteq \Zp^{d}$ so, by the strong triangle inequality, we have that
 \begin{equation*}
 |b_{i}|_{p}=|b_{0}|_{p} \max \left\{ |x_{i}|_{p},\left|\frac{b_{i}}{b_{0}}\right|_{p} \right\}=
 |b_{0}|_{p}\left| x_{i}-\frac{b_{i}}{b_{0}} \right|_{p}<p^{(n+m\lambda)/d}H^{-v_{i}}<H^{-1}
 \end{equation*}
for $H>H_0$. Such inequality fails unless $b_{i}=0$, since $|b_{i}| \leq H$. Thus, $\frac{b_{i}}{b_{0}} \in \Zp$ for all $1 \leq i \leq d$.

Now we are ready to construct $(a_0,\dots,a_n)$ with $(a_{0},p)=1$. Let $k\ge0$ be the unique integer such that $p^{k}|b_{0}$ but $p^{k+1} \nmid b_{0}$. Then, since $\frac{b_{i}}{b_{0}} \in \Zp$ so we have that $p^{k}|b_{i}$ for all $1 \leq i \leq d$. By \eqref{ineq2}, we get that
 \begin{align*}
 |b_{d+j}|_p &\le \max\left\{\left| b_0f_{j}(\bx)-\sum_{i=1}^{d}\frac{\partial f_{j}}{\partial x_{i}}(\bx) \left(b_0x_{i}-b_{i}\right)-b_{d+j}\right|_{p}, |b_0f_{j}(\bx)|_p,\left|\sum_{i=1}^{d}\frac{\partial f_{j}}{\partial x_{i}}(\bx) \left(b_0x_{i}-b_{i}\right)\right|_p\right\}\\[2ex]
  & \le \max\left\{H^{-\tau_{j}},p^{-k},p^\lambda p^{(n+m\lambda)/d}H^{-v_{\min}} \right\}=p^{-k},
 \end{align*}
since $\tau_j>1$ and $H>H_0$. Therefore, $p^k|b_{d+j}$ and we have that $\frac{b_{d+j}}{b_{0}} \in \Zp$ for each $1 \leq j \leq m$. In particular we have that $d=\gcd(b_{0},b_{1},\dots, b_{n})=p^k$. For $0 \leq i \leq n$ define the numbers $a_{i}=p^{-k}b_{i}$, which, by what we have proven above, are all integers satisfying $\gcd(a_{0},a_{1},\dots, a_{n})=1$ and, by the choice of $k$, $(a_{0},p)=1$.
By the third inequality of \eqref{ineq2}, we have that $\max_{0 \leq i \leq n} |a_{i}| \leq p^{-k}H$, which verifies the third inequality in \eqref{tau_{2}}. Further, using the first set of inequalities of \eqref{ineq2}, we get that
 \begin{equation}\label{eqn45}
 |a_{0}x-a_{i}|_{p} = |p^{-k}b_{0}x-p^{-k}b_{i}|_{p} =p^{k}|b_{0}x-b_{i}|_{p}
 <p^{(n+m\lambda)/d} p^kH^{-v_{i}}
 \end{equation}
for each $1 \leq i \leq d$, since $v_i>1$. This verifies the first set of inequalities in \eqref{tau_{2}}.

By \eqref{eqn45} and the fact that $p^k\le H$, we get that
$$
\left( \frac{a_{1}}{a_{0}}, \dots , \frac{a_{d}}{a_{0}} \right)\in B\left( \bx, p^{(n+m\lambda)/d}H^{-v_{\min}+1} \right) \subseteq B(\bx,\varepsilon)\subseteq  \U,
$$
where the last inclusion follows from condition $(\beta)$ on $H_{0}$. Thus, $\by=\left( \frac{a_{1}}{a_{0}}, \dots , \frac{a_{d}}{a_{0}} \right) \in \U$ and, in particular,  $f_{j}\left(\frac{a_{1}}{a_{0}}, \dots , \frac{a_{d}}{a_{0}}\right)$ is well defined and \eqref{eqn007} is applicable to $f=f_j$ for each $1 \leq j \leq m$.

Using the fact that each $f_{j}$ is \DQE{} at $\bx$ we get that
\begin{align} \label{f_tau}
\left| f_{j} \left(\frac{a_{1}}{a_{0}}, \dots , \frac{a_{d}}{a_{0}} \right)-f_{j}(\bx)-\sum_{1 \leq i \leq d} \frac{\partial f_{j}}{\partial x_{i}}(\bx) \left(\frac{a_{i}}{a_{0}}-x_{i} \right) \right|_{p} & < C\max_{1 \leq i \leq d} \left| \frac{a_{i}}{a_{0}}-x_{i} \right|_{p}^{2}\\
& < (p^{-k}H)^{-\tau_{j}} \nonumber
\end{align}
for each $1 \leq j \leq m$, where the last inequality follows since
\begin{align*}
C\max_{1 \leq i \leq d} \left| \frac{a_{i}}{a_{0}}-x_{i} \right|_{p}^{2} & \;\stackrel{\eqref{eqn45}}{<}\; C p^{(2n+2m\lambda)/d} p^{2k}H^{-2v_{\min}}\\
& = C p^{(2n+2m\lambda)/d}p^{-2k(v_{\min}-1)} (p^{-k}H)^{-2v_{\min}}\\[1ex]
& \stackrel{(*)}{\le} (p^{-k}H)^{-\tau_{\max}} \le (p^{-k}H)^{-\tau_{j}}\,.
\end{align*}
Here $(*)$ follows from condition $(\alpha1)$ on $H_{0}$ if $p^k\le H_0^{1/2}$ and it follows from condition $(\alpha2)$ on $H_{0}$ if $p^k>H_0^{1/2}$, and we also use the facts that $v_{\min}>1$ and $2v_{\min}>\tau_{\max}$.

For each $1 \leq j \leq m$ in the second row of inequalities of \eqref{ineq2} we may divide through by $p^k=|b_{0}|^{-1}_{p}$ and $p^\lambda$, and combine with \eqref{f_tau} to obtain
\begin{equation*}
\left| f_{j} \left( \frac{a_{1}}{a_{0}}, \dots , \frac{a_{d}}{a_{0}} \right)-\frac{a_{d+j}}{a_{0}} \right|_{p}< (p^{-k}H)^{-\tau_{j}}
\end{equation*}
for each $1 \leq j \leq m$. This verifies the second set of inequalities in \eqref{tau_{2}}, while the first set of inequalities in \eqref{tau_{2}} follows from \eqref{eqn45}. The proof is thus complete.
\end{proof}

In order to use a Mass Transference Principle, namely Theorem~\ref{MTPRR_full_measure}, we now establish the following Corollary.

\begin{corollary} \label{coldiri}
Let $\bff$, $\btau$ and $\bv$ be as in Theorem \ref{diri}. Let $\bx \in \U \backslash \Q^{d}$ and $\lambda$ be given by \eqref{first_order_constant}. Then the following system
\begin{equation} \label{cas}
\begin{cases}
\left|x_{i} -\frac{a_{i}}{a_{0}}\right|_{p} < p^{(n+m\lambda)/d}h^{-v_{i}} & (1 \leq i \leq d), \\[2ex]
\left|f_{j}\left(\frac{a_{1}}{a_{0}}, \dots , \frac{a_{d}}{a_{0}} \right) - \frac{a_{d+j}}{a_{0}} \right|_{p} < h^{-\tau_{j}} \, \qquad& (1 \leq j \leq m)\,,
\end{cases}
\end{equation}
where $h=\underset{0 \leq i \leq n}{\max} |a_{i}|$, has infinitely many integer solutions $(a_{0}, \dots, a_{n}) \in \Z^{n+1}$ satisfying \eqref{a}.
\end{corollary}

\begin{proof}
First, observe that \eqref{cas} is a consequence of \eqref{tau_{2}} since $h=\max_{0 \leq i \leq n} |a_{i}| \leq p^{-k}H$ and $v_i>1$ for all $i$. So we only need to verify that there are infinitely many different solutions $(a_{0}, \dots, a_{n}) $ to \eqref{tau_{2}} as $H$ varies. Suppose the contrary. Then, since $\bx \in \Zp^{d} \backslash \Q^{d}$, there is $1 \leq i \leq d$ such that $x_{i}-\frac{a_{i}}{a_{0}} \neq 0$ and so
\begin{equation} \label{delt}
\delta:=\min \left|x_{i}-\frac{a_{i}}{a_{0}}\right|_{p}>0
\end{equation}
where the minimum is taken amongst the solutions $(a_{0}, a_{1},\dots, a_{n})$ to \eqref{tau_{2}} over all $H\ge H_0$.
On the other hand, by \eqref{tau_{2}}, we have that $\delta< p^{(n+m\lambda)/d}p^{k}H^{-v_i}\le p^{(n+m\lambda)/d}H^{-v_i+1}\to0$ as $H\to\infty$ since $v_i>1$, giving a contradiction for large $H$.
\end{proof}

\begin{corollary} \label{coldiri+}
Let $\bff$, $\btau$ and $\bv$ be as in Theorem \ref{diri} and suppose that $\bff$ is \DQE{} for almost every $\bx\in\U$. Let $\delta>0$ be any constant. Then for almost every $\bx \in \U$ the following system
\begin{equation} \label{cas+}
\begin{cases}
\left|x_{i} -\frac{a_{i}}{a_{0}}\right|_{p} < \delta h^{-v_{i}} & (1 \leq i \leq d), \\[2ex]
\left|f_{j}\left(\frac{a_{1}}{a_{0}}, \dots , \frac{a_{d}}{a_{0}} \right) - \frac{a_{d+j}}{a_{0}} \right|_{p} < h^{-\tau_{j}} \, & (1 \leq j \leq m)\,,
\end{cases}
\end{equation}
where $h=\underset{0 \leq i \leq n}{\max} |a_{i}|$, has infinitely many integer solutions $(a_{0}, \dots, a_{n}) \in \Z^{n+1}$ satisfying \eqref{a}.
\end{corollary}

\begin{proof}
Define the set of integer points
\begin{equation}\label{eqn047}
S_{\btau}=\left\{(a_{0},\dots, a_{n}) \in \Z^{n+1}:\, \begin{array}{l}
\text{\eqref{a} holds and for all }1 \leq j \leq m\\[1ex]
\left| f_{j}\left(\frac{a_{1}}{a_{0}}, \dots , \frac{a_{d}}{a_{0}} \right) -\frac{a_{d+j}}{a_{0}}\right|_{p}<h^{-\tau_{d+j}}, \\[1.5ex]
\text{where }\underset{0 \leq i \leq n}{\max} |a_{i}| = h
\end{array}
\right\},
\end{equation}
and for each $\ba\in S_{\btau}$ and $\delta>0$ consider the hyperrectangles
\begin{equation}\label{eqn048}
B_{\ba}(\btau;\delta)=\left\{ \bx \in \Zp^{d}: \left|x_{i}-\frac{a_{i}}{a_{0}}\right|_{p}<\delta h^{-\tau_{i}} \qquad (1 \leq i \leq d ) \right\}.
\end{equation}
By Corollary~\ref{coldiri},
\begin{equation}\label{eqn050}
\bigcup_{\delta>0}\underset{\ba \in S_{\btau}}\limsup B_{\ba}(\btau;\delta)=\U\setminus\Q_p^d
\end{equation}
and therefore this union has full measure in $\U$, since the sequence of sets in \eqref{eqn050} is increasing as $\delta$ increases. These are Borel sets and therefore measurable. Hence, by the continuity of measure, we have that
\begin{equation}\label{eqn051}
\lim_{\delta\to+\infty}\mu_{p, n}\left(\underset{\ba \in S_{\btau}}\limsup B_{\ba}(\btau;\delta)\right)=
\mu_{p, n}\left(\bigcup_{\delta>0}\underset{\ba \in S_{\btau}}\limsup B_{\ba}(\btau;\delta)\right)=
\mu_{p, n}(\U)\,.
\end{equation}
By Lemma~\ref{measure_unchanged_rectangles}, every limsup set in \eqref{eqn051} is of the same measure. Hence,
$$
\mu_{p, n}\left(\underset{\ba \in S_{\btau}}\limsup B_{\ba}(\btau;\delta)\right)=
\mu_{p, n}(\U)
$$
for every $\delta>0$. This is exactly what we had to prove.
\end{proof}

\section{Proof of Theorems \ref{lowerbnd}--\ref{lowerbnd_manifold}}

We begin with the following proposition that lays the basis for applying the Mass Transference Principles.

\begin{proposition}\label{prop0}
Let $\bff: \U \to \Zp^{m}$, where $\U\subseteq \Zp^{d}$ is an open subset, and for $\bx\in\U$ let $\ff(\bx)=(\bx,\bff(\bx))$. Let $\U^*$ be the subset of $\bx\in\U$ such that $\bff$ is \DQE{} at $\bx$.
Let  $\btau=(\tau_{1},\dots, \tau_{n})\in\R^n_+$. Let $S_{\btau}$ and $B_{\ba}(\btau;\delta)$ be defined by \eqref{eqn047} and \eqref{eqn048} respectively. Then for any $0<\delta\le 1$
\begin{equation}\label{eqn47}
\U^*\cap\underset{\ba \in S_{\btau}}\limsup B_{\ba}(\btau;\delta)\subset \ff^{-1}(\Wp_n(\btau))
\end{equation}
provided that
\begin{equation}\label{eqn48}
\min_{1 \leq i \leq d} \tau_{i} > \max_{1 \leq j \leq m}\tau_{d+j}\,.
\end{equation}
If
\begin{equation}\label{eqn49}
\min_{1 \leq i \leq d} \tau_{i} = \max_{1 \leq j \leq m}\tau_{d+j}\,.
\end{equation}
and $\bff$ is a Lipschitz map with the Lipschitz constant $L$, then \eqref{eqn47} holds for any $0<\delta\le \min\{1,L^{-1}\}$.
\end{proposition}

\begin{proof}
Suppose $\bx\in \U^*\cap B_{\ba}(\btau;\delta)$. Then
\begin{align*}
\left|f_{j}(\bx)-f_{j}\left(\frac{a_{1}}{a_{0}}, \dots , \frac{a_{d}}{a_{0}} \right)\right|_{p} & <
\max\left\{\max_{1\le i\le d}\left|\frac{\partial f_j(\bx)}{\partial x_i}\right|_p\max_{1 \leq i \leq d}\left|x_{i}-\frac{a_{i}}{a_{0}} \right|_{p},C\max_{1 \leq i \leq d}\left|x_{i}-\frac{a_{i}}{a_{0}} \right|_{p}^2 \right\}\\[2ex]
 & <
\max\left\{\max_{1\le i\le d}\left|\frac{\partial f_j(\bx)}{\partial x_i}\right|_p \delta h^{-\tau_{\min}},C\delta^2 h^{-2\tau_{\min}} \right\}
 < h^{-\tau_{d+j}}
\end{align*}
for any $1 \leq j \leq m$ and all sufficiently large $h$ if \eqref{eqn48} holds. In turn, if \eqref{eqn49} holds, we use the fact that $\bff$ is Lipschitz:

\begin{align*}
\left|f_{j}(\bx)-f_{j}\left(\frac{a_{1}}{a_{0}}, \dots , \frac{a_{d}}{a_{0}} \right)\right|_{p} & <
 L\max_{1 \leq i \leq d}\left|x_{i}-\frac{a_{i}}{a_{0}} \right|_{p} < L\delta h^{-\tau_{\min}}  \le h^{-\tau_{d+j}}
\end{align*}
for any $1 \leq j \leq m$ and all sufficiently large $h$ since $0<\delta\le L^{-1}$.
In either case, if $\ba\in S_{\btau}$, then
$$
\left|f_{j}(\bx)-\frac{a_{d+j}}{a_0}\right|_{p}\le \max\left\{\left|f_{j}(\bx)-f_{j}\left(\frac{a_{1}}{a_{0}}, \dots , \frac{a_{d}}{a_{0}} \right)\right|_{p},
\left|\frac{a_{d+j}}{a_0}-f_{j}\left(\frac{a_{1}}{a_{0}}, \dots , \frac{a_{d}}{a_{0}} \right)\right|_{p}
\right\}  <
 h^{-\tau_{d+j}}
$$
provided that $h$ is sufficiently large. Hence, assuming that
$\bx\in \U^*\cap \underset{\ba \in S_{\btau}}\limsup B_{\ba}(\btau;\delta)$ we conclude that the system of inequalities
\begin{align}\label{wt1}
\begin{cases}
|a_{0}x_{i}-a_{i}|_{p}<\delta h^{-\tau_{i}}\le h^{-\tau_{i}}, \quad& (1 \leq i \leq d), \\
|a_{0}f_{j}(\bx)-a_{d+j}|_{p}< h^{-\tau_{d+j}} & (1 \leq j \leq m), \\
\max \{|a_{0}|,\dots,|a_{n}| \} =h
\end{cases}
\end{align}
holds for infinitely many $\ba\in\Z^{n+1}$. Therefore,
$\bx\in \ff^{-1}(\Wp_n(\btau))$ and the proof is complete.
\end{proof}

\begin{proof}[Proof of Theorems \ref{lowerbnd}--\ref{lowerbnd1}]
First of all, note that \eqref{vb1} and \eqref{eqn:lwrbnd} follow from Theorem~\ref{lowerbnd_manifold}.
Thus we only need to verify the measure part of these theorems, that is \eqref{vb1+} and \eqref{vb1+3}. Consequently, we will assume that $\bff$ is Lipschitz on $\U$. Let $0<\delta\le \min\{1,L^{-1}\}$, where $L$ is the Lipschitz constant of $\bff$.
With reference to the General Mass Transference Principle (Theorem~\ref{general_MTP}),
take the function $g(x)=x^{d}$ as our dimension function. Note that $g$ is doubling and that $\ha^{g}\asymp \mu_{p,d}$. For any ball $B=B(x,r)$ and dimension function $f(x)=x^{s}$, define $B^{s}= B(x,g^{-1}(x^{s}))$. Note that in Theorems~\ref{lowerbnd} and \ref{lowerbnd1} we have that $\tau_{1}=\tau_{2}= \dots=\tau_{d}$. Therefore
the sets $B_{\ba}(\btau;\delta)$ defined by \eqref{eqn048} are balls. Let the vector $\bv=(v_{1}, \dots , v_{d})$ be of the form $\bv=(v, \dots , v)$ where
\begin{equation*}
v= \frac{n+1-\sum_{i=1}^{m}\tau_{d+i}}{d}.
\end{equation*}
Note that this $\bv$ satisfies the requirements of Theorem \ref{diri} and its corollaries. Let
\begin{equation*}
s=\frac{n+1-\sum_{i=1}^{m}\tau_{d+i}}{\tau_{d}},
\end{equation*}
Then
\begin{equation*}
B^{s}_{\ba}(\tau_{d};\delta)= \left\{ \bx \in \Zp^{d} : \max_{1 \leq i \leq d}\left|x_{i}-\frac{a_{i}}{a_{0}}\right|_{p}<\delta^{s/d}h^{-v} \right\},
\end{equation*}
and, by Corollary \ref{coldiri+},
\begin{equation*}
 \mu_{p,d}\left(\underset{\ba \in S_{\btau}}{\limsup}B^{s}_{\ba}(\tau_{d};\delta) \right)=\mu_{p,d}(\U).
 \end{equation*}
  Hence, for any ball $B \subset \U$,
\begin{equation*}
\ha^{g}\left( B \cap \limsup_{\ba \in S_{\btau}}B^{s}_{\ba}(\tau_{d};\delta) \right) = \ha^{g}\left(B \right).
\end{equation*}
By the Mass Transference Principle (Theorem~\ref{general_MTP}), we have that for any ball $B \subseteq \U$,
\begin{equation} \label{hausdorff_result}
\ha^{s}\left( B \cap \limsup_{\ba \in S_{\btau}}B^{g}_{\ba}(\tau_{d};\delta) \right) = \ha^{s}\left(B \right).
\end{equation}
By Proposition~\ref{prop0} and the choice of $\delta$, we have that \eqref{eqn47} holds, where $\U^*=\U$.
Combining \eqref{hausdorff_result} and \eqref{eqn47} gives the required Hausdorff measure results and completes the proof.
\end{proof}

\vspace*{1ex}

\begin{proof}[Proof of Theorem \ref{lowerbnd_manifold}]
First of all, without loss of generality we can assume throughout this proof that \eqref{eqn48} holds. Otherwise we could consider $\btau'=(\tau_1+\varepsilon,\dots,\tau_d+\varepsilon,\tau_{d+1},\dots,\tau_{n})$ for a suitably small $\varepsilon>0$ and note that $\ff^{-1}\left(\Wp_n(\btau')\right)\subset \ff^{-1}\left(\Wp_n(\btau)\right)$. Hence, the validity of \eqref{vb1+4} for $\btau'$ would give us the bound
$$
\dim \left(\ff^{-1}\left(\Wp_n(\btau)\right)\right) \ge \dim \left(\ff^{-1}\left(\Wp_n(\btau')\right)\right) \geq \min_{1 \leq i \leq d}\left\{ \frac{n+1+\sum_{\tau_{j} < \tau_{i}}(\tau_{i}-\tau_{j})}{\tau_{i}+\varepsilon}-m \right\}
$$
and on letting $\varepsilon\to0$ we would get the required result for $\btau$.

Now, since \eqref{eqn48} holds, by Proposition~\ref{prop0} with $\delta=1$, get that
\begin{equation}\label{eqn47++}
\underset{\ba \in S_{\btau}}\limsup B_{\ba}(\btau;1)\subset \ff^{-1}(\Wp_n(\btau))
\end{equation}

Corollary~\ref{coldiri+} provides us with a full measure statement, which will be the basis for applying the Mass Transference Principle from rectangles to rectangles without Ubiquity (Theorem \ref{MTPRR_full_measure}). With reference to the notation used in Theorem \ref{MTPRR_full_measure} take
\begin{equation*}
\begin{array}{ll}
J=S_{\btau}, & \rho(q)=q^{-1},\\[1ex]
R_{\alpha}=\left\{ \left(\frac{a_{1}}{a_{0}}, \dots , \frac{a_{n}}{a_{0}}\right)\right\},\quad & \beta_{\alpha}=a_0\qquad\text{for }\alpha=(a_0,\dots,a_n) \in S_{\btau}
\end{array}
\end{equation*}
and so
\begin{equation}\label{eqn64}
\limsup_{\ba \in S_{\btau}}B_{\ba}(\bv;1)= \limsup_{\alpha \in J} \Delta(R_{\alpha},\rho(\beta_{\alpha})^{-\bv}).
\end{equation}
By Corollary \ref{coldiri+} and \eqref{eqn64}, we have that
\begin{equation}\label{eqn:66}
\mu_{p,d}\left(\limsup_{\alpha \in J} \Delta(R_{\alpha},\rho(\beta_{\alpha})^{-\bv}) \right)=\mu_{p,d}(\U)
\end{equation}
for any $\bv=(v_{1}, \dots , v_{d}) \in \R^{d}_{+}$ satisfying
\begin{equation} \label{bv_conditions}
v_{i}>1, \quad \sum_{i=1}^{d}v_{i}=n+1-\sum_{j=1}^{m}\tau_{j}.
\end{equation}
Without loss of generality we will assume that $\tau_{1} > \tau_{2} > \dots > \tau_{d}$. Similarly to what proceeds the proof of Proposition \ref{ubiquity_rectangle} define each $v_{i}$ recursively, starting with $r=0$, by
\begin{equation*}
v_{d-r}=\min\left\{ \tau_{d-r}, \frac{n+1-\sum_{j=1}^{m}\tau_{d+j}-\sum_{i=d-r+1}^{d}v_{i}}{d-i} \right\}.
\end{equation*}
Observe that this choice of $\bv$ satisfies \eqref{bv_conditions}. Furthermore, there exists a $1 \leq b \leq d$ such that
\begin{equation*}
v_{c}=\frac{n+1-\sum_{j=1}^{m}\tau_{d+j}-\sum_{i=d-b}^{d}v_{i}}{d-b}
\end{equation*}
for all $1 \leq c \leq d-b$. Define $t_1,\dots,t_d$ from the equations
\begin{equation*}
\tau_{j}=v_{j}+t_{j}
\end{equation*}
then note that $\bt=(t_{1}, \dots , t_{d}) \in \R^{d}_{\geq0}$ and thus satisfies the conditions of Theorem \ref{MTPRR_full_measure}. Thus, the set $W(\bt)$, defined in Theorem~\ref{MTPRR_full_measure}, is exactly the right hand side of \eqref{eqn47++}. Hence, by \eqref{eqn47++}, we get that
$$
\dim \ff^{-1}(\Wp_n(\btau))\ge \dim W(\bt)\,.
$$
Also, in view of \eqref{eqn:66}, Theorem~\ref{MTPRR_full_measure} is applicable and so $\dim \ff^{-1}(\Wp_n(\btau))\ge s$, where $s$ is the same as in Theorem~\ref{MTPRR_full_measure}. The proof is now split into the following three cases.

\begin{enumerate}[i)]
\item $A_{i} \in \{ v_{1}, \dots v_{d-b} \}$: For these values of $A_{i}$, which are defined in Theorem~\ref{MTPRR_full_measure}, we have that
\begin{equation*}
\begin{array}{ccc}
K_{1}=\{1, \dots , d-b \}, & K_{2}=\{ d-b+1, \dots , d\}, & K_{3}=\emptyset.
\end{array}
\end{equation*}
Applying Theorem \ref{MTPRR_full_measure} gives
\begin{align*}
\dim \ff^{-1}(\Wp_n(\btau))\ge \dim W(\bt) & \geq \min_{1 \leq i \leq d-b} \left\{ \frac{(d-b)v_{i}+(d-(d-b+1)+1)v_{i} -\sum_{j=d-b}^{n}t_{j}}{v_{i}} \right\}, \\
&=\min_{1 \leq i \leq d-b} \left\{ d-\frac{\sum_{j=d-b+1}^{d}t_{j}}{v_{i}} \right\}. \\
\end{align*}
Since $t_{i}=0$ for $d-b+1 \leq i \leq d$ we have that $\dim \ff^{-1}(\Wp_n(\btau))\geq d$, which is the maximal possible dimension for $\ff^{-1}(\Wp_n(\btau))$.
\item $A_{i} \in \{ v_{d-b+1}, \dots , v_{d} \}$: For such values of $A_{i}$ observe that
\begin{equation*}
\begin{array}{ccc}
K_{1}=\{1, \dots , i \}, & K_{2}=\{i+1, \dots , d\}, & K_{3}=\emptyset.
\end{array}
\end{equation*}
Then in this case we have that
\begin{equation*}
\dim \ff^{-1}(\Wp_n(\btau))\ge \dim W(\bt) \geq \min_{d-b+1 \leq i \leq d}\left\{ \frac{iv_{i} + (d-i)v_{i} - \sum_{j=i+1}^{d} t_{j}}{v_{i}}  \right\}.
\end{equation*}
Similarly to the previous case, since $t_{j}=0$ for $d-b+1 \leq j \leq d$ the r.h.s of the above equation is $d$.
\item $A_{i} \in \{\tau_{1}, \dots , \tau_{d}\}$: Since $\tau_{i}=v_{i}$ for $d-b+1 \leq i \leq d$, ii) covers such result. So we only need to consider the set of $A_{i} \in \{ \tau_{1}, \dots \tau_{d-b} \}$. If $A_{i}$ is contained in such set, then
\begin{equation*}
\begin{array}{ccc}
K_{1}=\emptyset, & K_{2}=\{ i, \dots , d\}, & K_{3}=\{ 1, \dots , i-1\}.
\end{array}
\end{equation*}
Thus, by Theorem \ref{MTPRR_full_measure}, we have that
\begin{align*}
& \dim \ff^{-1}(\Wp_n(\btau))\geq \min_{1 \leq i \leq d} \left\{ \frac{(d-i+1)\tau_{i}+\sum_{j=1}^{i-1}v_{j}-\sum_{j=i}^{d}t_{j}}{\tau_{i}} \right\}, \\
&= \min_{1 \leq i \leq d} \left\{ \frac{(d-i+1)\tau_{i}+(i-1)\left( \frac{n+1-\sum_{j=1}^{m}\tau_{d+j}-\sum_{j=d-b+1}^{d}v_{j}}{d-b} \right)-\sum_{j=i}^{d-b}(\tau_{j}-v_{j})-\sum_{j=d-b+1}^{d}t_{j}}{\tau_{i}} \right\}, \\
& = \min_{1 \leq i \leq d}\left\{ \frac{(d-i+1)\tau_{i}+(d-b)\left( \frac{n+1-\sum_{j=1}^{m}\tau_{d+j}-\sum_{j=d-b+1}^{d}v_{j}}{d-b} \right)-\sum_{j=i}^{d-b}\tau_{j}-\sum_{j=d-b+1}^{d}t_{j}}{\tau_{i}} \right\}, \\
&=\min_{1 \leq i \leq d} \left\{ \frac{n+1+\sum_{j=i}^{d}(\tau_{i}-\tau_{j})-\sum_{j=1}^{m}\tau_{d+j}}{\tau_{i}} \right\}, \\
&=\min_{1 \leq i \leq d}\left\{ \frac{n+1+\sum_{j=i}^{n}(\tau_{i}-\tau_{j})}{\tau_{i}}-m \right\}.
\end{align*}
\end{enumerate}
Considering all cases we have that
\begin{equation*}
\dim \ff^{-1}(\Wp(\btau) )\geq \dim W(\bt) \geq \min_{1 \leq i \leq d}\left\{ \frac{n+1+\sum_{j=i}^{n}(\tau_{i}-\tau_{j})}{\tau_{i}}-m \right\}
\end{equation*}
as required.
\end{proof}

\end{document}